\documentclass[11pt]{amsart}
\usepackage{subfigure}
\usepackage{graphicx}
\usepackage{stmaryrd}
\usepackage{textcomp}
\usepackage{enumerate}
\usepackage{setspace}
\usepackage{amssymb}
\usepackage{amsthm}
\usepackage{amsmath}
\usepackage{graphicx}
\usepackage{extarrows}
\usepackage{marvosym}
\usepackage{empheq}
\usepackage{latexsym}
\usepackage{endnotes}
\usepackage{fontenc}
\usepackage{color}
\usepackage{comment}

\allowdisplaybreaks

\begin{document}

	\newtheorem{theorem}{Theorem}
	\newtheorem{definition}[theorem]{Definition}
	\newtheorem{question}[theorem]{Question}
	
	\newtheorem{lemma}[theorem]{Lemma}
	\newtheorem{claim}[theorem]{Claim}
	\newtheorem*{claim*}{Claim}
	\newtheorem{corollary}[theorem]{Corollary}
	\newtheorem{problem}[theorem]{Problem}
	\newtheorem{conj}[theorem]{Conjecture}
	\newtheorem{fact}[theorem]{Fact}
	\newtheorem{observation}[theorem]{Observation}
	\newtheorem{prop}[theorem]{Proposition}

	\newcommand\marginal[1]{\marginpar{\raggedright\parindent=0pt\tiny #1}}
	
	\theoremstyle{remark}
	\newtheorem{example}[theorem]{Example}
	\newtheorem{remark}[theorem]{Remark}
	\newcommand{\RED}{\color{red}}
	
	\newcommand{\supp}{\mathrm{supp}}
	\def\eps{\varepsilon}
	\def\HH{\mathcal{H}}
	\def\E{\mathbb{E}}
	\def\C{\mathbb{C}}
	\def\R{\mathbb{R}}
	\def\Z{\mathbb{Z}}
	\def\N{\mathbb{N}}
	\def\PP{\mathbb{P}}
	\def\T{\mathbb{T}}
	\def\l{\lambda}
	\def\s{\sigma}
		\def\vp{\varphi}
	\def\t{\theta}
	\def\a{\alpha}
	\def\la{\langle}
	\def\ra{\rangle}	
	\def\endproof{{\hfill $\square$} }
	\def\Xt{\widetilde{X}}
	\def\Pt{\widetilde{P}}
	\def\Var{\mathrm{Var}}
	\def\PV{\mathrm{PV}}
	\def\NV{\mathrm{NV}}
	\def\NN{\mathcal{N}}
	\def\CC{\mathcal{C}}
	
	\def\cA{\mathcal{A}}
	\def\cQ{\mathcal{Q}}	
	\def\cC{\mathcal{C}}
	\def\cG{\mathcal{G}}
	\def\F{\mathcal{F}}
	\def\tm{\tilde{\mu}}
	\def\ts{\tilde{\sigma}}
	\def\L{\Lambda}
	\def\z{\zeta}
		\def\p{\partial}
	\def\g{\gamma}
	\def\Q{\mathcal{Q}}
	\newcommand{\Cov}{\mathrm{Cov}}
	\renewcommand{\P}{\mathbb{P}}
	\newcommand{\Corr}{\mathrm{Corr}}
	\newcommand{\SCS}{\mathrm{SCS}}
	\newcommand{\SCC}{\mathrm{SCC}}
	\newcommand{\one}{\mathbf{1}}
	\newcommand{\bx}{\mathbf{x}}
	\newcommand{\by}{\mathbf{y}}
		\newcommand{\bu}{\mathbf{u}}
		\newcommand{\bv}{\mathbf{v}}
	\newcommand{\bw}{\mathbf{w}}
	\newcommand{\bz}{\mathbf{z}}
	\newcommand{\bt}{\mathbf{t}}
	\newcommand{\mes}{\mathrm{mes}}
	\newcommand{\tr}{\mathrm{tr}}
	\newcommand{\Exp}{\mathrm{Exp}}
		\def\im{\mathrm{Im}}
		\def\re{\mathrm{Re}}
\pagestyle{plain}

\title{Random polynomials: the closest roots to the unit circle}
\author{Marcus Michelen \and Julian Sahasrabudhe}
\begin{abstract}
Let $f = \sum_{k=0}^n \eps_k z^k$ be a random polynomial, where $\eps_0,\ldots ,\eps_n$ are iid standard Gaussian random variables,
and let $\z_1,\ldots,\z_n$ denote the roots of $f$. 
We  show that the point process determined by the magnitude of the roots $\{ 1-|\z_1|,\ldots, 1-|\z_n| \}$  tends to 
a Poisson point process at the scale $n^{-2}$ as $n\rightarrow \infty$. One consequence of this result is that it determines the magnitude of the closest root to the unit circle. In particular, we show that 
\[ \min_{k} ||\z_k| - 1|n^2 \rightarrow \Exp(1/6),\]
in distribution, where $\Exp(\l)$ denotes an exponential random variable of mean $\l^{-1}$. This resolves a conjecture of Shepp and Vanderbei from 1995 that was later studied by Konyagin and Schlag.
\end{abstract}

\maketitle

\section{Introduction}

We consider the typical distribution of the zeros of the polynomial
\begin{equation} \label{eq:def-f} f(z) = \sum_{k=0}^n \eps_k z^k, \end{equation}
where $\eps_0,\ldots,\eps_n$ are iid standard Gaussian random variables. 
This problem originates in the 1930s with seminal works of Bloch and P\'olya \cite{bloch-polya}, Littlewood and Offord, \cite{littlewood-offord1,littlewood-offord2,littlewood-offord4,littlewood-offord3} and Erd\H{o}s and Offord \cite{erdosOfford} and, in the years since, many aspects of this problem have come to be well understood
(see, for example, \cite{ibragimov-maslova4,ibragimov-maslova1,ibragimov-maslova3,ibragimov-maslova2,peres-virag,rice1,rice2,tao-vu}). 

One immediately striking aspect of the zeros of a random polynomial is that they cluster tightly and uniformly around the unit circle. This phenomenon, now widely known, was first discovered by \v{S}paro and \v{S}ur \cite{sparoSur} in the 1960s and is now known to persist for a wide variety of coefficient distributions, thanks to the work of Arnold \cite{arnold} and Ibragimov and Zaporozhets \cite{IZ}. Finer aspects of this convergence are also known: the proportion of roots that are within $\rho/n$ of the unit circle was determined by Shepp and Vanderbei \cite{shepp-vanderbei} and the limiting distribution of the roots that have \emph{constant} distance away from the unit circle, was determined\footnote{Peres and Vir\'ag actually work in a slightly different setup: for them the $\eps_j$ are iid standard \emph{complex} Gaussian random variables, a case which it is easy to adapt our results to.} in the celebrated work of Peres and Vir\'ag \cite{peres-virag}.

In this paper we determine the \emph{microscopic} nature of the clustering around the unit circle by studying the point process determined by the roots at distance $O(n^{-2})$ from the unit circle, the scale at which the first zeros appear. To lay this out a little more carefully, let us order the roots of $f$ according to their distance from the unit circle 
\[ |1-|\z_1|| \leq |1-|\z_2|| \leq \cdots \leq |1-|\z_n|| \]
and write $1-|\z_i| = x_i/n^{2}$. In their 1995 paper, Shepp and Vanderbei \cite{shepp-vanderbei} conjectured that the closest root to the unit circle is at distance $\Theta(n^{-2})$ and that the point process determined by the roots at this distance  $\{x_1,x_2,x_3,\ldots \}$ is asymptotically a Poisson point process. In this paper we prove these conjectures. For this, let $f \sim \cG_n$ denote the probability space defined by \eqref{eq:def-f}.

\begin{theorem}\label{thm:main1} If $f \sim \cG_n$ then the set 
\[ \{ (|\zeta| - 1)n^2 : f(\zeta) = 0 \} \] converges to a homogeneous Poisson process with intensity $1/12$, as $n\rightarrow \infty$, in the vague topology.
\end{theorem}

From this, we can immediately resolve the conjecture of Shepp and Vanderbei regarding the scale at which the first zeros appear. One direction of this conjecture was already proved by Konyagin and Schlag \cite{konyagin-schlag} who showed that the closest zero to the unit circle is at distance \emph{at least} $cn^{-2}$ with positive probability. Here we provide a matching upper bound and, in fact, determine the asymptotic law of this distribution. To state this result we write $Z(f)$ for the set of complex zeros of $f$, let 
$\mathbb{S}$ denote the unit circle in the complex plane, let $d(A,B) = \inf\{ |x-y| : x \in A, y\in B\}$ for sets $A,B \subseteq \C $
and, for $\l \geq 0$, let $\Exp(\l)$ denote an exponential random variable with mean $\l^{-1}$.

\begin{corollary}\label{cor:main}
Let $f \sim \cG_n$ then
\[ n^2d(Z(f),\mathbb{S}) \rightarrow \Exp(1/6) ,\] in distribution, as $n \rightarrow \infty$.
\end{corollary}

Note that Corollary~\ref{cor:main} is indeed a corollary of our Theorem~\ref{thm:main1}, and requires no more information about polynomials: given a homogeneous Poisson point process $x_1,x_2,\ldots $ on $\R$, we have that $\min\{|x_1|,|x_2|,\ldots \}$ is exponentially distributed.  

It also appears that our techniques are strong enough to prove a slightly stronger and intuitive generalization of Theorem \ref{thm:main1}: the two-dimensional point process \[ \{ \left( (|\zeta|-1)n^2, \arg(\zeta)\right)   : f_n(\zeta) = 0 ,\arg(\zeta)\in[0,\pi]\}\] converges to a homogeneous Poisson process on\footnote{Note here that we restrict to roots $\z$ with $\arg(\zeta) \in [0,\pi]$, since the roots with 
$\arg(\z) \in [\pi,2\pi]$ are simply a reflection of this set, due to the fact $f$ has real coefficients.} the strip $\R\times [0,\pi]$ of intensity $1/(12\pi)$.  In this paper, however, we limit ourselves to providing a proof of Theorem~\ref{thm:main1}.

An interesting point of contrast comes from another conjecture made by Shepp and Vanderbei \cite{shepp-vanderbei} who also considered the closest 
\emph{real} root to the unit circle. In this direction, they conjectured that the closest real root is at distance $\Theta(n^{-1})$ and the point process determined by the real roots at distance $\Theta(n^{-1})$ converges to a Poisson process as $n \to \infty$. In contrast with Theorem \ref{thm:main1}, the first named author will show in a forthcoming paper \cite{MMrealroots} that the limiting point process converges to a point process which is \emph{not} Poisson and will also affirm the conjecture of Shepp and Vanderbei, that the closest real roots to the unit circle appear at this scale. Indeed, in \cite{MMrealroots} it is shown that with probability $1-\eps$,
\[ d(Z(f) \cap \R, \mathbb{S}) = \Theta_{\eps}(n^{-1}),\]
for all $\eps  >0$. 

\subsection{A heuristic and proof sketch}

The main new idea behind the proof of Theorem~\ref{thm:main1} is conceptually simple and we hope it will inspire uses beyond the results of this paper. To get a feel for this idea, let us start by considering the zero set
\[\{ z :  \re\, f(z) = 0 \}  \] in an annular neighbourhood of the unit circle. 
As we will see, this set appears as $\approx n$ curves in this annular region (which we will call \emph{strings}, to borrow a phrase from combinatorial geometry) which begin somewhere inside the unit circle and then cross over to the outside of the unit circle. For the purposes of this discussion, it is also convenient to imagine these strings as coloured \emph{red}. Likewise we imagine the zero set
\[ \{ z : \im\, f = 0 \}, \]
as $\approx n$ \emph{blue} strings which behave in much the same fashion. Now while we will have essentially no crossings between strings of the same colour, crossings between strings of \emph{different} colours correspond exactly to the zeros of $f$ in our annular neighbourhood. To find a root of $f$ near the unit circle, our strategy will be to find a pair of strings, one blue and one red, that are extremely close to each other \emph{on} the unit circle. We will then see (with high probability) that these strings must cross \emph{near} the unit circle, thereby giving us our zero of $f$.

This idea, assuming it can be made rigorous, reduces the problem to showing that there exist two strings, of different colours, that get quite ``close'' on the unit circle, meaning (as we'll see), with distance $O(n^{-2})$. For this, consider the trigonometric polynomials
\[ X(x) := \sum_{k=0}^n \eps_k \cos kx = \re\, f(e^{ix}), \qquad Y(x) := \sum_{k=0}^n \eps_k \sin kx = \im\,f(e^{ix}), \]
and observe that we would like to show that there is a pair of roots $(x_0,y_0)$ of $X,Y$, respectively, with $|x_0- y_0| = O(n^{-2})$.

To see that this is a reasonable goal, let $R$ be the set of (red) zeros of $X$ in $[0,\pi]$ and $B$ be the set of (blue) zeros of $Y$ in\footnote{We ignore $[\pi,2\pi]$ as the zeros here are simply a refection of the zeros in $[0,\pi]$.} $[0,\pi]$. A classical result, due to Dunnage \cite{dunnage}, tells us that $|R|,|B| \sim n/\sqrt{3}$. So if $R,B$ were each sets of $n/\sqrt{3}$ \emph{iid} points in $[0,\pi]$, then we would have $d(R,B) \approx n^{-2}$, as desired.

While this heuristic appears promising, one aspect of the (true) distribution of the zeros of the real and imaginary parts of $f$ seem to point in a different direction: the roots of the real and imaginary parts of $f$ actually
\emph{repulse} each other (see Figure~\ref{figure}). Thus one may be lead to believe that the phenomena described above actually \emph{fails} when applied to roots, as opposed to random sets of points. However, as we shall show, the behavior of the distribution of the real roots $x,y$ with $d(x,y) = \Theta(n^{-2})$ remains Poisson in the real case, albeit with a different parameter.

\begin{figure}[h] \label{figure}
	\centering
	\subfigure[The zeros of $\re f$ and $\im f$ on the unit circle \newline  for $n = 40$ in red and blue respectively.] {
		\includegraphics[width=2.7in]{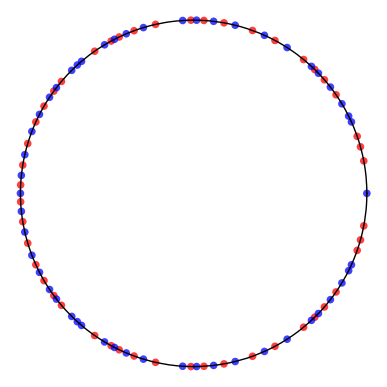} 
	\hspace{.5in}}
	\subfigure[The same number of red and blue  points placed independently at random.] {
		\includegraphics[width=2.7in]{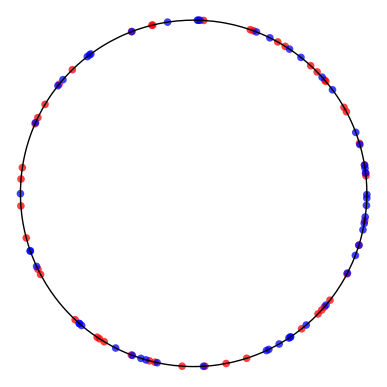}}
	\caption{}
\end{figure}

Turning this heuristic into a proof is rather involved and consists of two main steps: first, we show (in Section \ref{sec:unit-circle-redux}) that to understand the point process corresponding to roots $\{\zeta\}$ of $f$ with $|\zeta| = 1 + \Theta(n^{-2})$, it is sufficient to study a different point process $\mu_f$ defined entirely in terms of the behavior of $f$ on the unit circle.  Roughly speaking, $\mu_f(U)$ is the number of pairs of zeros of $\re(f)$, $\im(f) $ that have the correct position and velocity
so that their associated strings will collide at some radius $r \in U$. The second step consists of showing that our new process $\mu_f$ converges to a Poisson point process.  To do this, we use the method of moments along with a Kac-Rice formula which allows us to express the factorial moments of $\mu_f(U)$ as an integral of a certain kernel. 
We then work with this kernel: for tuples of zeros that are far apart we shall show that this kernel approximately factors, which roughly says that the behavior of far away roots is independent.

The main challenge lies in dealing with the case when the roots are clustered together (in various possible configurations) and thus are significantly dependent. Our main tool here is Lemma~\ref{th:density-bound}, which is our main technical contribution of this paper and consumes most of its length.
 We rephrase this lemma here in a slightly different way to give the reader a feel for the strength of the result.
 
\begin{lemma}\label{lem:density-bound-restatement} Let $I_1,\ldots,I_{2k}$ be disjoint arcs of the unit circle above the real line which satisfy 
$d(I_i, \{\pm 1\}) >n^{-1/2}$. For $f \sim \cG_n$, let $A_i$ be the event that 
$\re(f)$ has a root in $I_i$ and $B_i$ be the event that $\im(f)$ has a root in $I_{i+ k}$. Then 
\begin{equation}\label{eq:PAB} \PP\Big( \bigcap_{i=1}^k (A_i \cap B_i) \Big) = O_k\Big( n^{2k} \prod_{j=1}^{2k} |I_{j}|   \Big). \end{equation}
\end{lemma}

This result is sharp, up to constants, for all sizes of intervals $I_1,\ldots,I_k$, and thus gives very good control even when $|I_i|$  is much smaller than $n^{-1}$. Observe in the statement of Lemma~\ref{lem:density-bound-restatement}, we specify that the intervals are above the real axis. This is because the roots of $f$ are identical above and below the real axis. We also specify that these intervals are not too close to the real axis. Indeed, some condition of this form is required
as $f$ has a very different behavior (again due to its real coefficients) near $\pm 1$. For an extreme example, we point out that $\im(f)$ is \emph{always} zero on the real line and thus $\im(f)$ always has a zero in any interval $I$ containing $\{\pm 1\}$, which would be detrimental to an estimate like \eqref{eq:PAB}. 

We also point out that Lemma~\ref{th:density-bound} (or, equivalently, Lemma~\ref{lem:density-bound-restatement}) can be seen as extending a key lemma of Granville and Wigman \cite[Proposition A.1]{granville-wigman} to its logical conclusion. Granville and Wigman, prove a variant of Lemma~\ref{th:density-bound} for three zeros in a single interval of length $O(n^{-1})$. While there are some similarities in the approach, our generalization is not at all straightforward.

\subsection{Future research}

It appears that the notion of studying the zero sets $\{ z : \re\, f = 0 \}$ and $\{ z : \im\, f = 0 \}$ for a random polynomial is novel and many 
natural questions suggest themselves about the nature of these ensembles of ``strings''. Most generally, one might ask if there is a natural probabilistic notion that models these ensembles of strings. It is also natural to ask if other phenomena, such as the fascinating results of Peres and Vir\'ag \cite{peres-virag}, have a pleasing explanation in terms of these trajectories. 

It also appears natural to consider the microscopic structure of the zero set of a random polynomial about other circles  $\{ z : |z| = r \}$, where $r \geq 1$. For $r = 1+O(n^{-1})$ we would expect a very similar behavior to what we see around $|z|=1$, but with the ``force'' of the repulsion increasing with $r$. It seems particularly interesting to consider the distribution of roots about the circles with radius $r = 1 +  \omega(n^{-1})$, where (we would imagine) the effects of the repulsion start to seriously warp the distribution.

Another direction would be to consider variants of Theorem~\ref{thm:main1} for different coefficient distributions. While we have not investigated this question here, we would imagine that the behavior seen in Theorem~\ref{thm:main1} is ``universal'' in the sense that a similar result should remain true for a wide class of coefficient distributions. We offer the following as a target for future research.

\begin{conj}
	Let $\{\eps_j\}_j$ be iid real random variables with $\E\, \eps_1 = 0$ and $\E\, \eps_1^2 = 1$.  Then the conclusion of Theorem \ref{thm:main1} still holds. 
\end{conj}

Perhaps the most natural first step in this direction would be to extend Theorem~\ref{thm:main1} to the case of random \emph{Littlewood} polynomials: polynomials where the $\eps_k$ are chosen in $\{\pm 1\}$ independently and uniformly.

\section{Reduction to the unit circle}\label{sec:unit-circle-redux}
The purpose of this section is to make rigorous a central piece of the heuristic outlined in the introduction.  In particular, we show that to understand the zeros of $f_n$ near the unit circle, it is sufficient to look at zeros of the real and imaginary part on the unit circle along with their derivatives at those points.

But before getting to this, we get an irritating matter out of the way: since $f$ has real coefficients, we have $|\overline{\z}| = |\z|$ for all the roots $\z$ of $f$ and thus each distance in the 
sequence $(|\z_1|-1)n^2,\ldots,(|\z_n|-1)n^2$ occurs twice for roots $\z$ with $\z \in \C \setminus \R$. To sweep away this redundancy, we consider only the roots in the upper half plane; that is, with $\im\, \z \geq 0 $. 
Now, for $S \subseteq \R$ and $n \in \N$, define the annulus in the upper-half plane
\[ \cA_n(S) = \{ z \in \C : (|z|-1)n^2 \in S \textit{ and } \im\, \z \geq 0 \}. \] For a polynomial $f$ with $\deg(f) = n$, we define
\[ \cA_f(S) = \{ z \in \cA_n(S) : f(\z) = 0  \}, \]
 and define the measure $\nu_f$ on $\R$ by
$$\nu_{f}(S) = |\cA_f(S)|,$$
for each Borel set $S \subseteq \R$.  The measure $\nu_f$ is our main object of interest and in fact Theorem \ref{thm:main1} is exactly the statement that the random counting measure $\nu_f$ converges to a Poisson point process as $n \to \infty$. 

To study $\nu_f$, we show in this section that it is sufficient to work with another measure $\mu_{f}$, which is easier to work with and is defined solely in terms of the behavior of $f$ on the unit circle.  Throughout, we write $\T := \R/(2\pi\Z )$ but often just work in $[0,\pi]$. Since $f$ behaves differently near the real axis than elsewhere, it will be convenient for us to work only with points away from the real axis; with this in mind, define 
\[ \T_{0} := \{x \in [0,\pi]: d(x,\{0,\pi\}) > n^{-1/2} \}. \]
Now, to define the measure $\mu_f$, break $f$ into real and imaginary parts
\[ f((1+\rho)e^{ix}) = X(x,\rho) + iY(x,\rho), \] where
\[ X(x,\rho) = \sum_{j = 1}^n \eps_j (1+\rho)^j \cos(jx),\,\, \text{  and   }\, \, Y(x,\rho) = \sum_{j = 1}^n \eps_j (1+\rho)^j \sin(jx)\,.\]
We also define  $X(x) := X(x,0)$ and similarly for $Y$.  For a Borel set $S \subseteq \R$, we define $\cC_f(S)$ to be 
\begin{align} 
\left\{ (x,y) \in \T_0^2:\, X(x) = Y(y) = 0 ,\frac{(x-y)X'(x)Y'(y)n^2}{(X'(x))^2 + (Y'(y))^2 } \in S, |x - y| \leq n^{-2}(\log n)^4 \right\} \, \label{eq:c-def}
\end{align}
and then set
\[ \mu_f(S) := |\cC_f(S)|. \]
Roughly speaking we have designed the measure of $S$ to be the number of pairs of zeros, of $X$ and $Y$ respectively, that are at the right distance from each other and moving at the right ``speed'' (as $\rho$ changes)
so that they will result in a zero $\z$ of $f$ with $(|\z|-1)n^2 \in S$. 

The following lemma, to which the remainder of this section is dedicated, makes the connection between $\nu_f$ and $\mu_f$ rigorous.

\begin{lemma}\label{lem:colliding-roots}
	Let $f \sim \cG_n$ and let $I \subseteq \R$ be a bounded open interval. Then 
	$$\P\left(\nu_f(I) = \mu_f(I) \right) \to 1\, ,$$ as $n$ tends to infinity.
\end{lemma}

The key idea behind the proof of Lemma~\ref{lem:colliding-roots}, is that every zero $\z$ of $f$ that is \emph{close} to the unit circle can be traced back to two nearby zeros \emph{on} the unit circle of $\re(f)$ and $\im(f)$, respectively. This is established in the following two sister lemmas (Lemmas~\ref{lem:pairing1}, \ref{lem:pairing2}), the proofs of which are quite similar yet different enough in a few important ways and so we have treated them separately.  We now turn to state a few standard facts that we will make heavy use of.

\begin{fact} \label{facts} 	\mbox{}
	\begin{enumerate}
		\item \textbf{Salem-Zygmund Inequality}: For $f \sim \cG_n$ there is a constant $C > 0$ so that we have $$\max_{|z| = 1}|f(z)| \leq C (n \log n)^{1/2},$$
		with high probability.
		
		\item \textbf{Bernstein's Inequality}: If $g$ is a polynomial of degree $n$ then $$\max_{|z| = 1} |g'(z)| \leq n \max_{|z| = 1} |g(z)|\,.$$
		
		\item \label{eq:fact-max3} If $g$ is a polynomial of degree $n$ then for each $\rho \geq 1$ we have $$\max_{|z| = \rho} |g(z)| \leq \rho^n \max_{|z| = 1} |g(z)|\,.$$
	\end{enumerate}
\end{fact}

\vspace{3mm}

\noindent We shall also make use of the following properties of Gaussian random variables.

\begin{fact} \label{fact:gaussian} For $\s > 0$, let $X \sim N(0,\sigma^2)$ be a centered Gaussian random variable. 
\begin{enumerate}
\item \label{item:gaussMoment} For all $k \in \N$, there exists a constant $C_k$, independent of $\s$, for which $\E |X|^k = C_k \sigma^k$.
\item \label{item:gaussCondition} If $(X,Y)$ is a bivariate Gaussian random variable and $y\in \R$  then $\Var(X \vert \, Y = y) \leq \sigma^2$.
\item \label{item:gaussSmallBall} If $U \subseteq \R$ is an open set then $\PP(X \in U ) = O(|U|/\sigma)$. 
\end{enumerate}
\end{fact}

Let us say that $\z \in \C$ and $(x,y) \in \T^2$ are $\delta$-close if $|\z - e^{ix}|, |\z - e^{iy}| < \delta$.  The following lemma allows us to associate a zero of $f$ to a $\delta$-close pair of zeros of $X$ and $Y$, assuming a regularity hypothesis is met.

\begin{lemma}\label{lem:pairing1}
Put $\delta := n^{-2}(\log n)^3$, let $I \subseteq \R$ be an open, bounded interval. Then there exists $n_0 = n_0(I)$ for which the following holds for all $n > n_0(I)$. Let $f$ be a polynomial with 
\begin{equation}\label{eq:f-not-too-big} \max_{x \in [0,2\pi]} |f(e^{ix})| \leq  n^{1/2}\log n.\end{equation}
If $\z \in \cA_f(I)$ with $\arg(\zeta) = \theta \in \T_0$ and 
\begin{equation} \label{eq:non-degen}|X'(\t)|,|Y'(\t)| > n^{3/2}/\log n \end{equation}
then $\z$ is $\delta$-close to a pair $(x_0,y_0) \in \cC_f(I)$
\end{lemma}

\begin{proof}
We have that $\z$ is a root of $f$ where $\zeta = \rho_0e^{i\t} = (1 + \gamma/n^2)e^{i\theta}$ for some $\gamma \in I$ and $\theta \in \T_0$. We first select $x_0,y_0$ and then show that they satisfy the conclusions of Lemma~\ref{lem:pairing1}.

We express $X,Y$ using the Taylor expansion in variables $\t,x$, at $x = \t, \rho = 0 $,
\begin{equation}\label{eq:exp-X}  X(x,\rho) = X(\t) + X'(\t)(x-\t) + Y'(\t)\rho + O(n^{-3/2} (\log n)^3 ), \end{equation}
\begin{equation}\label{eq:exp-Y}   Y(y,\rho) = Y(\t) + Y'(\t)(y-\t) - X'(\t)\rho + O(n^{-3/2} (\log n)^3 ),  \end{equation}
where the error bound holds when $|x-\t|, |y-\t|, \rho = O(n^{-2}\log n)$. One can see this last point by bounding the second derivatives of $X$,$Y$:
\[ |X_{xx}(x,\rho)|,|Y_{xx}(x,\rho)| \leq \sup_{x \in [0,2\pi]} |f^{(2)}((1+\rho)e^{ix})| \leq 2n^2\sup_{x \in [0,2\pi]} |f(e^{ix})|  = O(n^{5/2}\log n).\]
Here we have used Bernstein's inequality along with \eqref{eq:fact-max3} in Fact~\ref{facts} and our assumption \eqref{eq:f-not-too-big}.
   
Now, since $\z \in Z(f)$, we have $X(\t,\rho_0) = Y(\t,\rho_0) = 0$ which we write out as 
\[  0 = X(\t,\rho_0) = X(\t) + Y'(\t)\rho_0 + O(n^{-3/2} (\log n)^3 );\]
\[  0 = Y(\t,\rho_0) = Y(\t) - X'(\t)\rho_0 + O(n^{-3/2} (\log n)^3 ). \]
We then choose $x_0,y_0$, so that $X'(\t)(x_0-\t) = Y'(\t)\rho_0$ and $Y'(\t)(y_0 -\t) = -X'(\t)\rho_0$,
up to lower-order terms. That is, there exist $x_0,y_0$ so that $X(x_0) = Y(y_0) = 0$ and 
\begin{equation} \label{eq:x-t} x_0 = \theta + \frac{\gamma}{n^2} \frac{Y'(\t)}{ X'(\t)} + O(n^{-3}(\log n)^4) ;  \end{equation} 
\begin{equation} \label{eq:y-t} y_0 = \theta - \frac{\gamma}{n^2}\frac{X'(\t)}{Y'(\t)} + O(n^{-3}(\log n)^4).\end{equation}
We now check that this choice of $x_0,y_0$ satisfies the conclusion of Lemma~\ref{lem:pairing1}.
Using \eqref{eq:non-degen}, note that this choice of $(x_0,y_0)$ is $\delta$-close to $\z$, as desired. 
Rearranging \eqref{eq:x-t} and \eqref{eq:y-t}, we may write 
\begin{equation}\label{eq:gam} \gamma = \frac{n^2(x_0 - y_0)X'(\t)Y'(\t) }{(X'(\t))^2 + (Y'(\t))^2} + o(1).\end{equation}
We now need to replace $X'(\t),Y'(\t)$, in \eqref{eq:gam}, with $X'(x_0),Y'(y_0)$ to fit the definition of $\cC_f(S)$. We achieve 
this with a easy application of the mean value theorem. Indeed, since $x_{0}$ is within $O(n^{-2}(\log n)^2)$ of $\t$ there exists a 
$\xi$ with $|\t-\xi| = O(n^{-2}(\log n)^2)$ so that 
\[ X'(x_0) = X'(\t) + X^{(2)}(\t)(\t-\xi). \]
Since $|X'(\t)|>n^{3/2-o(1)}$, by \eqref{eq:non-degen} and $|X^{(2)}(\t)| < n^{5/2+o(1)}$ (again using \eqref{eq:f-not-too-big} and Bernstein's inequality) we see $X'(x_0) = X'(\t)(1+n^{-1 +o(1)})$. Applying the same to $Y$ yields
\[  \frac{n^2(x_0 - y_0)X'(\t)Y'(\t) }{(X'(\t))^2 + (Y'(\t))^2} = \gamma + o(1), \]
and therefore the left-hand-side is in $I$ for sufficiently large $n$, since $I$ is an open interval and $\gamma \in I$.
From \eqref{eq:x-t}, \eqref{eq:y-t} and condition \eqref{eq:non-degen} we see that $|x_0 - y_0| = O(n^{-2}(\log n)^2)$.  Therefore $(x_0,y_0) \in \cC_f(I)$, for large enough $n$. \end{proof}

\vspace{4mm}

\noindent In a way very similar to the proof of Lemma~\ref{lem:pairing1}, we may track how the roots move as $\rho$ changes.
\begin{lemma}\label{lem:root-movement}
Let $f$ be a polynomial satisfying 
\[ \max_{x \in [0,2\pi]} |f(e^{ix})| \leq n^{1/2}\log n,\] and let
$x_0,y_0 \in \T$ be such that $X(x_0) = 0 $ and $Y(y_0) = 0$, and 
\[ |X'(x_0)|,|Y'(x_0)|, |X'(y_0)|,|Y'(y_0)| > n^{3/2}/\log n. \]
Then, for each $|\rho| \leq n^{-2}\log n$ there exists $x_\rho, y_{\rho} \in \T$ so that 
\[ X(x_{\rho},\rho) = 0, \qquad Y(y_{\rho},\rho) = 0, \]
where 
\[ x_{\rho} = x_0 - \rho \frac{Y'(x_0)}{X'(x_0)} + O\left(n^{-3}(\log n)^4\right), \]  
\[ y_{\rho} = y_0 + \rho \frac{X'(y_0)}{Y'(y_0)} + O\left( n^{-3}(\log n)^4 \right). \]
\end{lemma}

Note crucially,  that if the $x_0,y_0$ are very close to each other then $Y'(x_0)/X'(x_0) \approx Y'(y_0)/X'(y_0)$
and thus Lemma~\ref{lem:root-movement} tells us that as we increase $\rho$ the roots $x_\rho,y_{\rho}$ are moving in opposite directions on the circle.  This is, in fact, a direct consequence of the Cauchy-Riemann equations.  The next lemma associates a nearby complex root of $f$ to a pair of roots of $X,Y$ on the unit circle.

\begin{lemma}\label{lem:pairing2}
Put $\delta = n^{-2}(\log n)^3$, let $I \subseteq \R$ be an open interval. Then there exists $n_0 = n_0(I)$ for which the following holds for all $n > n_0(I)$. Let $f$ be a polynomial with 
\[ \max_{x \in [0,2\pi]} |f(e^{ix})| \leq n^{1/2}\log n.\]
If $(x_0,y_0) \in \cC_f(I)$ with
\[ |X'(x_0)|,|Y'(x_0)|, |X'(y_0)|,|Y'(y_0)| > n^{3/2}/\log n, \]
then there exists $\z \in \cA_f(I)$ which is $\delta$-close to $(x_0,y_0)$. 
\end{lemma}
\begin{proof}
We apply Lemma~\ref{lem:root-movement}, to see that $X(x_{\rho},\rho) = Y(x_{\rho},\rho) = 0 $ where
\begin{equation}\label{eq:xrho} x_{\rho} = x_0 - \rho \frac{Y'(x_0)}{X'(x_0)} + O(n^{-3}(\log n)^4), \end{equation} 
\begin{equation}\label{eq:yrho} y_{\rho} = y_0 + \rho \frac{X'(y_0)}{Y'(y_0)} + O(n^{-3}(\log n)^4 ). \end{equation}
To compare these two ratios, we apply the mean value theorem and use that $|X'(y_0)| > n^{3/2}/\log n$ and Bernstein's inequality to see that
for some $\xi$ with $|\xi -y_0| \leq |y_0 -x_0|$, we have
\[ X'(x_0) = X'(y_0) + X^{(2)}(y_0)(\xi-y_0) = (1+o(n^{-1/4}))X'(y_0),\] 
and likewise for $Y'(x_0)$. As a result, we have that 
\[ X'(y_0)/Y'(y_0) = (1+o(1))X'(x_0)/Y'(x_0).\]
Using this along with \eqref{eq:xrho}, \eqref{eq:yrho} and 
\[ (\log n)^2 > |X'(x_0)|/|Y'(x_0)| > 1/(\log n)^2 \]
we see that $x_{\rho},y_{\rho}$ are (up to lower order terms) traveling in opposite directions on the circle, and therefore we must have $x_{\rho_0} = y_{\rho_0}$ for some $\rho_0 = \gamma/n^2$ where
\[ \gamma = \frac{n^2(x - y)X'(x_0)Y'(y_0) }{(X'(x_0))^2 + (Y'(y_0))^2} + o(1). \]
Thus $\z := (1+\gamma/n^2)e^{i\t} \in \cA_f(I)$, for sufficiently large $n$, since $I$ is an open interval. Finally, we note that $(x_0,y_0)$ is $\delta = n^{-2}(\log n)^3$ close to $\z$, for large enough $n$. 
\end{proof}

\subsection{Dealing with pathological points}
Lemmas~\ref{lem:pairing1} and \ref{lem:pairing2} showed that we could pair each zero $\z$ of $f$ with a nearby root-pair of $X,Y$ (and vice versa), \emph{provided} our function was not doing something atypical around our zero. Here we record a few lemmas that say these atypical behaviors will
not be a problem for us. We postpone the fairly straightforward proofs of these lemmas to Appendix~\ref{sec:details}, so we don't distract from the main trajectory of our proof.  

\begin{lemma}\label{lem:good-complex} Let $f \sim \cG_n$.
	All $\zeta \in Z(f)$ with $||\zeta| - 1| \leq n^{-2}(\log n )^{1/4} $ satisfy 
	\[ |X'(\arg(\z))|,|Y'(\arg(\z))| > n^{3/2}/\log n,\] with high probability. 
\end{lemma}

\begin{lemma}\label{lem:circle-good}
	Let $f \sim \cG_n$ and let $I$ be a finite interval. Then all $(x_0,y_0) \in \mathcal{C}_f(I)$ satisfy
	$$
	|X'(x_0)|, |Y'(x_0)|, |X'(y_0)|, |Y'(y_0)| > n^{3/2} / \log n, 
	$$ with high probability.
\end{lemma}

For generic $z \in \C$, $f$ is a non-degenerate two-dimensional Gaussian.  However, near the real axis, the imaginary part of $f$ is small, thus leading to a different, more ``one-dimensional'' behavior that we will have to treat in a slightly different manner. In particular, we shall show that $f$ has no roots near the real axis that are within $O(n^{-2})$ of the unit circle.

\begin{lemma}\label{lem:endpoints} Let $f \sim \cG_n$, $M > 0$ and $\eps > 0$. Then $f$ has no zeros in $$\{z \in \cA_n([-M,M]) : |\arg(z)| \leq n^{-\eps} \text{ or }|\arg(-z)| \leq n^{-\eps}  \},$$
	with high probability.
\end{lemma}

The following lemma tells us that no two roots of $f$ have distance $\approx n^{-2+o(1)}$ in $\C$. 

\begin{lemma}\label{lem:close-roots} Let $f \sim \cG_n$ and 
	let $I$ be a bounded interval. Then, with high probability, there are no pairs of distinct $\z_1,\z_2 \in \cA_f(I)$ so that
	\[ |\zeta_1 - \zeta_2| \leq n^{-2}(\log n)^{10}. \]
\end{lemma}

Similarly, the following says that on the unit circle, neither $X$ nor $Y$ has two roots that are very close together, another example of the repulsion of roots of random polynomials.

\begin{lemma}\label{lem:close-roots-circle}
	Let $g$ be either $X$ or $Y$.  Then, with high probability, there do not exist distinct points $x_1,x_2 \in \T_0$ with $g(x_1) = g(x_2) = 0$, 
	\[|x_1 - x_2| \leq n^{-2} (\log n)^{10}\,.\]
\end{lemma}

\subsection{Proof of Lemma~\ref{lem:colliding-roots}}

\begin{proof}[Proof of Lemma~\ref{lem:colliding-roots}] As in Lemmas~\ref{lem:pairing1} and \ref{lem:pairing2}, we set $\delta = n^{-2}(\log n)^3$.
We define a function 
\[ \alpha : \cA_f(I) \rightarrow \cC_f(I),\] which is an injection with high probability.
Say that $\z = (1+\gamma/n^2)e^{i\t} \in \cA_{f}(I)$ is \emph{good} if 
\[|X'(\t)|,|Y'(\t)| > n^{3/2}/\log n, \]
and \emph{bad} otherwise. Given a good $\z \in \cA_f(I)$, we can apply Lemma~\ref{lem:pairing1} to see that there is a pair $(x_0,y_0) \in \cC_f(I)$, for large enough $n$, that is $\delta$-close to $\z$.
Define $\alpha( \z ) = (x_0,y_0)$. This defines $\alpha$ for all good $\z$. For bad $\z$, we define $\alpha$ to be an arbitrary point in $\cC_f(I)$. 

We need to show that there are no bad $\z$ with high probability and that there is no pair $(x_0,y_0)$ with $\alpha(\z) =\alpha(\z') =(x_0,y_0)$, for distinct good $\z,\z'$. Starting with this latter point, assume that both of $\z,\z'$ are mapped to the $\delta$-close pair $(x_0,y_0)$. This implies that $|\z-\z'| < 2n^{-2} (\log n)^3$, which occurs with probability $o(1)$, by Lemma~\ref{lem:close-roots}. 
We now turn to the former point and apply Lemma~\ref{lem:good-complex} to show that there are no bad $\z$ with high probability. Hence $\alpha$ is an injection with high probability and thus 
\[ \PP( |\cA_f(I)| \leq |\cC_f(I)|) = 1-o(1). \]

We now define a function $\beta : \cC_f(I) \rightarrow \cA_f(I)$, which will be an injection with high probability. For each $(x_0,y_0) \in \cC_f(I)$ we
say that $(x_0,y_0)$ is \emph{good} if 
\[ |X'(x_0)|,|Y'(x_0)|, |X'(y_0)|,|Y'(y_0)| > n^{3/2}/\log n,\]
and \emph{bad} otherwise. Now if $(x_0,y_0)\in \cC_f(I)$ is good then we may apply Lemma~\ref{lem:pairing2} to find a 
$\z \in \cA_f(I)$ that is $\delta$-close to $(x_0,y_0)$ and define $\beta((x,y)) = \z$. We define $\beta(x_0,y_0)$ to be an arbitrary element of $\cA$
if $(x,y)$ is bad. Lemma~\ref{lem:circle-good} tells us that there are no bad pairs with high probability. To see that this is an injection with high probability, assume that $\beta ((x_0,y_0)) = \beta((x_1,y_1)) = \z$. This means that both $(x_0,y_0),(x_1,y_1)$ are $\delta$ close to $\z$ and therefore
$|x_0 - x_1| < n^{-2}(\log n)^3$. Lemma~\ref{lem:close-roots-circle} tells us this happens with probability $o(1)$. Therefore $\beta$ is an injection with high probability and so
\[ \PP( |\cC_f(I)| \leq |\cA_f(I)|) = 1-o(1)\,, \]  
thus completing the proof of Lemma~\ref{lem:colliding-roots}. \end{proof}

\section{Moments and a Kac-Rice type formula}

With the work of Section~\ref{sec:unit-circle-redux} in hand, it is enough to show that the related measure $\mu_f$ converges to a Poisson point process. 
In this section we set up the remainder of the paper by expressing the moments of the random variables $\{\mu_f(U)\}_U$
in a convenient integral form, known as the Kac-Rice formula. 

But first we note the following lemma which tells us that to prove Theorem~\ref{thm:main1}, it is enough to study the factorial moments of the random variables $\mu_f(U)$, for all appropriate $U$. For a real number $x$, we use the notation
$(x)_k = x(x-1)\cdots (x-k+1)$.

\begin{lemma} \label{lem:method-of-moments}
Let $\{ m_n \}$ be a sequence of point processes on $\R$. Then $m_n$ converges in the vague topology to a Poisson point process of intensity $\l$ if  \begin{equation}\label{eq:moments-converge}
\E[(m_n(U))_k ] \to (\l |U|)^k,
\end{equation}
for all $U \subset \R$ and all $k \in \N$, where $U$ is a finite union of open intervals. Here $|U|$ denotes is the Lebesgue measure of $U$.
\end{lemma}
\begin{proof}
	By a theorem of Kallenberg \cite[Theorem 4.7]{kallenberg}, in order to show that $m_n$ converges to the Poisson process of intensity $\l$, it is sufficient to show that for each $a < b$ and $I$ of the form $(a,b]$ or $[a,b)$ we have that $\E\, m_n(I) \to \l|I|$ and $\P(m_n(I) = 0 )\to e^{-\l|I|}$.  By the method of moments \cite[Theorems 30.1 and 30.2]{billingsley}, both follow from showing convergence of the factorial moments; further, since $m_n$ asymptotically  assigns no mass to the endpoints, we may work with the interior of $I$.  Applying \eqref{eq:moments-converge} then shows that $m_n$ converges to the desired limit.  
\end{proof}

\vspace{4mm}

\noindent We now turn to our integral form for these moments. For this we define 
$$\vp_U(x,y):= \one\left\{ (x,y) \in \cC_f(U)  \right\}\, ,$$
where $f$ is a degree $n$ polynomial, $\cC_f(U)$ is as defined at \eqref{eq:c-def}, $x,y \in \T$ and $S \subseteq \R$. We now make an important, admittedly somewhat jarring, definition, the utility of which will be apparent soon.
For $k \in \N$, let $\bx \in \T^k$, $\by \in \T^k$ and define
\[ p_{k}(\bx,\by,U) :=  \frac{\E\left[\prod_{j = 1}^{k} |X'(x_j)|\cdot |Y'(y_j)| \cdot \vp_U(x_j,y_j) \,\big|\,\left\lbrace X(x_i) = Y(y_i) = 0  \right\rbrace_{i=1}^k \right]  }{(2\pi)^{k}\det(\Sigma)^{1/2}}, \] 
where $\Sigma = \Sigma_{k}(\bx,\by) = \Cov(X(x_i),Y(x_i))_{i \in [k]}$ is the covariance matrix of the joint distribution $(X(x_i),Y(x_i))_{i=1}^k$. We now arrive at our Kac-Rice-type integral for the factorial moments of $\mu_f(U)$.

\begin{lemma}\label{lem:our-KR-form}
	Let $U$ be a bounded open set and $k \in \N$.  Then 
\begin{equation} \label{eq:our-KR}
	\E[(\mu_f(U))_{k}] = 
\int_{\T_0^{2k} } p_{k}(\bx,\by,U)   \, d\bx d\by\,. \end{equation}
\end{lemma}

We postpone the proof of Lemma~\ref{lem:our-KR-form} to Appendix~\ref{sec:KR-details}, where we derive it from a very general form of the Kac-Rice integral.  

\vspace{4mm}
On a technical note, we point out that if $x_i = x_j$ for some $i \neq j$, or similarly for $\by$, then $\Sigma$ is singular and so $p_k$ is undefined.
In these cases, we just set $p_k = 0$, for completeness, although this is technically unnecessary, as we only care about the behavior of $p_k$ up to sets of measure $0$. 
As it turns out, this is ``the correct'' way of extending $p_k$, since $p_k$ tends to zero as $|x_i  - x_j| \to 0$, for some $i \not= j$, a fact which will be a consequence of our results.
 
\begin{remark}
	When we write $\E[|X'(x)| \,|\,X(x) = 0] $ we use the standard definition of a conditioned multivariate Gaussian.  This may be thought of as $\lim_{\eps\to0^+}\E[|X'(x)| \,|\,|X(x)| \leq \eps]$, which is different from $\lim_{\eps \to 0^+}\E[|X'(x)| \,|\, \exists~y\in[x-\eps,x+\eps] \text{ s.t. }X(y) = 0]\,$,
	which is another possible interpretation.
\end{remark}

\section{Using the integral form of the moments}\label{sec:using-KR}

In this section we set up the rest of the paper by introducing Lemma~\ref{th:density-bound}, which is our main technical lemma of the paper. After presenting this lemma, we will conclude the section by proving Theorem~\ref{thm:main1} assuming Lemma~\ref{th:density-bound}, thus providing motivation to our later sections.

In this direction, we introduce the density function
\begin{equation} \label{eq:def-pk} p_k(\bx,\by) := \frac{\E\left[\prod_{j = 1}^{k} |X'(x_j)|\cdot |Y'(y_j)|  \,\big|\, \left\lbrace X(x_i) = Y(y_i) = 0  \right\rbrace_{i=1}^k \right]  }{(2\pi)^{k}\det(\Sigma)^{1/2}}, \end{equation} which is a natural upper bound on $p_k(\bx,\by,U)$, for every $U$, and will be central to our discussion going forward. Our main technical lemma of this paper says that $p_k$ never gets too large. 

\begin{lemma}\label{th:density-bound}
For $k \in \N$, let $\bx,\by \in \T_0^k$ and $p_{k}$ be as above. Then
$$p_{k} (\bx,\by) \leq C_{k} n^{2k}\,,$$
where $C_{k}>0$ is a constant depending only on $k$.
\end{lemma}

In addition to being a natural upper bound for $p_k(\bx,\by,U)$, $p_k(\bx,\by)$ may be thought of as a density for $k$-tuples of zeros, a perspective we illustrate by deriving Lemma~\ref{lem:density-bound-restatement} from Lemma~\ref{th:density-bound}.

\begin{proof}[Proof of Lemma~\ref{lem:density-bound-restatement}]
	Set $S = \prod_i^{2k} I_i$ and bound $$\P\left(\bigcap_{i = 1}^k (A_i \cap B_i)\right) \leq \int_S p_{k}(\bx,\by)\,d\bx\,d\by \leq C_k n^{2k}|S| = C_k n^{2k} \prod_{j = 1}^{2k}|I_j|\,. $$
\end{proof}

\vspace{3mm}

We now prove Theorem~\ref{thm:main1}, assuming Lemma~\ref{th:density-bound}, by calculating the moments
of $\mu_f(U)$ using Lemma~\ref{th:density-bound}, for all open and bounded sets $U$. Recall, $\mu_f$ is the (random) measure defined in Section~\ref{sec:unit-circle-redux}.
As a first step in this direction, we lay out a few approximations that will be proved later. 

\begin{lemma}\label{lem:early-approx}
	For $d(\{x,y\},\pi\Z) \geq n^{-1/2}$ and $|x - y| \geq n^{-1/2}$ we have 
\begin{equation}\label{eq:early-approx1}	\Cov\left(\frac{X'(x)}{n^{3/2}},\frac{Y(y)}{n^{1/2}},\frac{Y'(y)}{n^{3/2}},\frac{X(x)}{n^{1/2}}\right) = (1+o(1)) \begin{bmatrix}
	\frac{1}{6} & -\frac{1}{4} & 0 & 0 \\
	-\frac{1}{4} & \frac{1}{2} & 0 & 0 \\
	0 & 0 & \frac{1}{6} & \frac{1}{4} \\
	0 & 0 & \frac{1}{4} & \frac{1}{2}
	\end{bmatrix}\,. \end{equation} We also have
	\begin{equation}\label{eq:early-approx2}
		\Cov\left(\frac{X'(x)}{n^{3/2}},\frac{Y'(y)}{n^{3/2}}\,\bigg|\,X(x) = Y(y)=0 \right) = (1+o(1)) \begin{bmatrix}
		\frac{1}{24} & 0 \\ 0 & \frac{1}{24}
		\end{bmatrix}\,.\end{equation}
	Further, if $F, G \in \{ X,X',Y,Y'\}$ then 
	\begin{equation}\label{eq:early-approx3} \Corr(F(x),G(y)) = O(n^{-1/2})\,.
	\end{equation} \end{lemma}
	
	\begin{proof} Lines \eqref{eq:early-approx1} and \eqref{eq:early-approx3} follow from Lemmas~\ref{lem:covar-approx} and Fact~\ref{fact:approx-int} while
	\eqref{eq:early-approx2} follows from \eqref{eq:early-approx1} and the formula for the variance of a conditioned Gaussian.\end{proof}

\vspace{4mm} 

\noindent We now consider our first moment $\E\, \mu_f(U)$, where $U$ is a finite union of intervals.

\begin{lemma}\label{lem:eval-mean} Let $U$ be an open set. We have that 
	\[ \E\, \mu_f(U)  = \frac{|U|}{12} + o(1) \,. \]
\end{lemma}
\begin{proof}
	We first apply Lemma~\ref{lem:our-KR-form} to express 
	\begin{equation} \label{eq:eval-mean1} \E\, \mu_f(U) = \int_{\T_0^2} p_1(x,y,U) \,d x d y\, 
	= \int_{\T_0}\int_{y: |x- y| \leq n^{-2}(\log n)^4 } p_1(x,y,U)\,dy\, dx , \end{equation}
	where the second equality holds since pairs $(x,y)$ with $|x - y| > n^{-2}(\log n)^4 $ have $p_1(x,y,U) = 0$. We now consider 
\begin{equation}\label{eq:p1} p_1(x,y,U) 
	= \frac{\E\left[| X'(x) Y'(y)| \cdot \vp_U(x,y) \,\big|\,  X(x) = Y(y) = 0   \right]  }{(2\pi)\det(\Sigma)^{1/2}}, \end{equation}
where $\Sigma$ is the 2 by 2 covariance matrix $\Cov( X(x),Y(y) )$.
We first directly compute the denominator of \eqref{eq:p1}, by using \eqref{eq:early-approx1} in Lemma~\ref{lem:early-approx}:
\begin{equation} \label{eq:detSigma} \det(\Sigma) = (1+o(1))\det \begin{bmatrix}
		\frac{1}{2} & 0 \\ 0 & \frac{1}{2}
		\end{bmatrix}\,=  (1+o(1))n^2/4.\end{equation}

Set $\sigma^2 = n^{3}/24$ and define the multivariate Gaussian 
\[ (W_1,W_2) = (X'(x)/\sigma,Y'(y)/\sigma \,|\,X(x) = Y(y) = 0) \] and note that $\E W_j^2 = 1 + o(1)$ by \eqref{eq:early-approx2} and $\Cov(W_1,W_2) = O(n^{-1/2})$ by \eqref{eq:early-approx3}.  We may then write $$ \E\left[| X'(x) Y'(y)| \cdot \vp_U(x,y) \,\big|\,  X(x) = Y(y) = 0   \right]  = \sigma^2 \E\left[|W_1 W_2| \one\left(\frac{n^2r W_1 W_2 }{W_1^2 + W_2^2} \in U\right) \right] $$
and note $$\E\left[|W_1 W_2| \one\left(\frac{n^2r W_1 W_2 }{W_1^2 + W_2^2} \in U\right) \right] = (1 + o(1))\E\left[|Z_1 Z_2| \one\left(\frac{n^2r Z_1 Z_2 }{Z_1^2 + Z_2^2} \in U\right) \right]$$
where $(Z_1,Z_2)$ is a standard $2$-dimensional Gaussian.  Putting the previous two lines together gives
%
\begin{equation} \label{eq:EX'Y'} \E\left[| X'(x) Y'(y)| \cdot \vp_U(x,y) \,\big|\,  X(x) = Y(y) = 0   \right] 
= (1+o(1))\frac{n^3}{24}\E\left[  |Z_1 Z_2|  \one\left(\frac{n^2 r Z_1 Z_2}{Z_1^2 + Z_2^2} \in U\right) \right]. \end{equation}

Now, with \eqref{eq:detSigma} and \eqref{eq:EX'Y'} in hand, we return to \eqref{eq:eval-mean1}. We set $\eta = n^{-2}(\log n)^4$ and, for each fixed $x$, we let $r = |x-y|$ and compute
\[  \int_{y: |x- y| \leq n^{-2}(\log n)^4 } p_1(x,y,U)\,dy = (1+o(1))\frac{n^2}{24\pi } \int_{-\eta}^{\eta}\E\left[|Z_1|\cdot |Z_2| \one\left(\frac{n^2 r Z_1 Z_2}{Z_1^2 + Z_2^2} \in U\right) \right]\,dr. \]
Anticipating an application of Fubini's theorem, note that $$\int_{-\eta}^\eta \one\left(\frac{n^2 r Z_1 Z_2}{Z_1^2 + Z_2^2} \in U \right)\,dr =\int_{-\eta}^\eta \one\left( r \in \left(\frac{Z_1^2 + Z_2^2}{n^2 Z_1 Z_2}\right)U\right)\,dr = \left|[-\eta,\eta ] \cap\left(\frac{Z_1^2 + Z_2^2}{n^2 Z_1 Z_2}\right)U\right|\,.$$

We then have \begin{align*}\int_{-\eta}^{\eta}\E\left[|Z_1|\cdot |Z_2| \one\left(\frac{n^2 r Z_1 Z_2}{Z_1^2 + Z_2^2} \in U\right) \right]\,dr  &= (1 + o(1)) \E\left[|U| n^{-2} \cdot |Z_1| \cdot |Z_2| \cdot \frac{Z_1^2 + Z_2^2}{|Z_1|\cdot |Z_2|} \right] \\
&= (1 + o(1)) 2|U|/n^2\,.\end{align*}
Combining then gives 
\[  \int_{y: |x- y| \leq n^{-2}(\log n)^4 } p_1(x,y,U)\,dy = (1 + o(1)) \frac{n^2}{24\pi}(2 \cdot |U| /n^2) =\frac{|U|}{12 \pi} +o(1). \]
Thus, from \eqref{eq:eval-mean1}, we have that 
\[ \E\, \mu_f(U) =\int_{\T_0} \int_{y: |x- y| \leq n^{-2}(\log n)^4 } p_1(x,y,U)\,dy\, dx = \frac{|U|}{12} +o(1), \]
as desired.
\end{proof}

\vspace{4mm}

\noindent We now show that $p_k(\bx,\by)$ approximately factors when the $x_j$ have pairwise distance at least $n^{-1/2}$. For this, we need an elementary fact
about inverting matrices in a neighborhood of the identity. Here and throughout, we use the notation $\|B\|_{op}$ to denote the  $\ell^2 \mapsto \ell^2$ operator norm.

\begin{fact}\label{fact:matrix-inverse} Let $M$ be a $d\times d$ matrix of the form $M =  I + E$, where $\|E\|_{op} = \delta$. If $\delta <1 $ then 
\[ M^{-1} =  I + B, \]
where $\|B\|_{op} \leq \delta/(1 - \delta)$.
\end{fact}
\noindent For what follows, let $k \in \N$ and define 
\[ D_{1} = D_{1,k,n} := \left\lbrace (\bx,\by) \in \T_0^{2k} : |x_i - x_j| \geq n^{-1/2} \textit{ for all } i\neq j \right\rbrace . \]

\begin{lemma}\label{lem:factor-density}
Let $U$ be an open set. For all $\bx \in D_{1,k}$ and all $\by$, we have
$$p_{k}(\bx,\by,U) = \left(1 + o(1)\right) \prod_{j = 1}^k p_1(x_j,y_j,U) + O(e^{-c(\log n)^2})\,.$$
\end{lemma}
\begin{proof} 
We consider the numerator and denominator of the density $p_{k}(\bx,\by,U)$ separately. That is,
\begin{equation}\label{eq:pkk} p_{k}(\bx,\by,U) = \frac{\alpha_{n,k}(\bx,\by)}{(2\pi)^{k} \det(\Sigma)^{1/2}}, \end{equation}
where 
\[ \Sigma = \Cov\left( X(x_j),Y(y_j) \right)_{i,j \in [k]}. \]

Note that we may restrict ourselves to considering $\by$ which satisfy $|y_j - x_j| \leq n^{-2}(\log n)^4$ for all $i \in [k]$, since $p_k = 0 $ otherwise. Now, for $j \in [k]$, define $\Sigma_j$ to be the $2\times 2$ covariance matrix $\Cov((X(x_j),Y(x_j)))$. We may use Lemma \ref{lem:early-approx} to see that 
$$
\Sigma = \begin{pmatrix}
\Sigma_1 & 0 & \cdots & 0 \\
0 & \Sigma_2 & \cdots & 0 \\
\vdots  & \vdots  & \ddots & \vdots  \\
0 & 0 & \cdots & \Sigma_k  
\end{pmatrix} + E 
$$
where $E$ is a matrix with all entries $O(n^{1/2})$.  Since the diagonal entries of each $\Sigma_j$ are $\Theta(n)$, it follows that
\begin{equation} \label{eq:det-factor}
\det \Sigma = (1 + O(n^{-1/2})) \prod_{j = 1}^k \det \Sigma_j \,.
\end{equation}

We now turn to the numerator of $p_{k}(\bx,\by,U)$,
$$\alpha_{n,k}(\bx,\by) = \E\left[\left(\prod_{j = 1}^k |X'(x_j)| \cdot |Y'(y_j)| \right)\cdot \varphi \,\bigg|\, X(x_j) = Y(y_j) = 0 \text{ for all  }j\right]\,.$$
Set 
\[ \psi := \one\left( |X'(x_i)| \leq n^{3/2}\log n, |Y'(x_j)| \leq n^{3/2}\log n  \text{ for all } i,j \in [k]\right)\]
and notice that we may introduce $\psi$ into the expectation in $\alpha_{n,k}$ and incur an additive error of at most $O(e^{-c (\log n)^2 })$. That is,
\begin{align}
\alpha_{n,k}(\bx,\by) &= \E\left[\left(\prod_{j = 1}^k |X'(x_j)| \cdot |Y'(y_j)| \right)\cdot \varphi\cdot \psi \,\bigg|\, X(x_j) = Y(y_j) = 0 \text{ for all  }j\right] + O(e^{-c(\log n)^2}) \nonumber\\
&= \int_{\T_0^{2k}} \frac{  s_1\ldots s_k t_1\ldots t_k\cdot \varphi\cdot \psi}{(2\pi)^k \det(M)^{1/2}}\exp\left(-(\mathbf{s},\mathbf{t})^T M^{-1}(\mathbf{s},\mathbf{t})/2\right)\, d\mathbf{s}\,d\mathbf{t} + O(e^{-c(\log n)^2}),\label{eq:alpha-as-int}
\end{align}
where we defined the covariance matrix $$M:= \Cov\left((X'(x_j),Y'(y_j))_{j = 1}^k \,|\, X(x_j) =  Y(y_j) = 0 \text{ for all  }j\right)$$ 
and re-written $(\mathbf{s},\mathbf{t})$ in the permutation $(s_1,t_1,s_2,t_2,\ldots,s_k,t_k)$.  

To finish the lemma, we need to show that we can ``replace'' the occurrences of $M$ in \eqref{eq:alpha-as-int} with $M'$
where
$$M' := \mathrm{diag}\left(\Cov((X'(x_j),Y'(y_j))\,|\,X(x_j) = Y(y_j) = 0)_{j=1}^k \right)\,.$$
This is easily done as $M,M'$ are both approximately multiples of the identity.
Indeed, we see that 
\[ M,M' = \frac{n^3}{24}I + O(n^{5/2}), \]
by using \eqref{eq:early-approx2} for the diagonal entries and \eqref{eq:early-approx3} for the off-diagonal entries.
It follows that $\det(M) = (1+O(n^{-1/2}))\det(M')$. To replace the occurrence of $M^{-1}$ in \eqref{eq:alpha-as-int}, we apply Fact~\ref{fact:matrix-inverse} to see that for all $(\mathbf{s},\mathbf{t})$ with
 $|(\mathbf{s},\mathbf{t})|_{\infty} < n^{3/2}\log n$, we have
$$|(\mathbf{s},\mathbf{t})^TM^{-1}(\mathbf{s},\mathbf{t}) - (\mathbf{s},\mathbf{t})^T(M')^{-1}(\mathbf{s},\mathbf{t})| 
= O( n^{-1/2}(\log n)^2)\,. $$
This shows that on the support of the integrand of \eqref{eq:alpha-as-int}, we may replace both instances of $M$ with $M'$ at the cost of a multiplicative error of at most $\left(1+O(n^{-1/2}(\log n)^2\right)$ and therefore  
$$\alpha_{n,k}(\bx,\by) = (1 + O(n^{-1/2}(\log n )^2)\prod_{j = 1}^k \E \left[ |X'(x_j)| \cdot |Y'(y_j)| \cdot \varphi \,|\, X(x_j) = Y(y_j) = 0 \right] + O(e^{-c(\log n)^2})\,.$$
Combining this with \eqref{eq:det-factor} completes the proof.\end{proof}

\vspace{3mm}
We now arrive at the main application of Lemma~\ref{th:density-bound} in the proof of Theorem~\ref{thm:main1}. This lemma deals the case when there are two root-pairs,  $(x_0,y_0)$ and $(x_1,y_1)$ of $X,Y$ that are close to each other.  In particular, define
\[ D_{0} = D_{0,k,n} := \{(\bx,\by) \in \T_0^{2k} : |x_i - x_j| \leq n^{-1/2} \text{ for some }i \neq j \}.\]

\begin{lemma}\label{lem:density-diagonal}For $k \in \N$, let $U \subseteq \R$ be open and bounded. We have
\[ \int_{D_0} p_{k}(\bx,\by,U) \,d\bx\, d\by  = o(1), \]
as $n$ tends to infinity.
\end{lemma}
\begin{proof}
	From the definition of $p_{k}$, we see that $p_{k} = 0$ if $|x_i-y_i| > n^{-2}(\log n)^4$ for some $i \in [k]$. So it makes sense to consider the integral over the set 
\[ D' := D_0 \cap \{(\bx,\by) : |x_j - y_j| \leq n^{-2}(\log n)^4 \text{ for all }j \}\,. \]
By Lemma~\ref{th:density-bound}, we have $p_{k}(\bx,\by,U) = O(n^{2k})$ and so
\[ I := \int_{D_0} p_{k}(\bx,\by,U)\,d\bx\, d\by =  \int_{D'} p_{k}(\bx,\by,U)\,d\bx \, d \by = O\left( |D'| \cdot n^{2k}\right). \]
Since 
\[ |D'| = O\left(\frac{(\log n)^{4k}}{n^{2k}} \frac{1}{n^{1/2}}\right), \] it follows that $I = o(1)$, as desired.\end{proof}

\vspace{3mm}

\noindent We now prove Theorem~\ref{thm:main1}, assuming Lemma~\ref{th:density-bound}.
\begin{proof}[Proof of Theorem~\ref{thm:main1}]
Let $U \subseteq \R$ be open and bounded and let $D_{1},D_0$ be as above. We have 
\[ \E[(\mu_f(U))_{k}] = 
\int_{\T_0^{2k}} p_{k}(\bx,\by,U)\,d\bx\,d\by    = \int_{D_0}p_{k}(\bx,\by,U)\, d\bx\,d\by   + \int_{D_1}p_{k}(\bx,\by,U)\, d\bx\,d\by\,. \]
By Lemma~\ref{lem:density-diagonal}, the first integral on the right-hand-side is $o(1)$, while we may apply Lemma~\ref{lem:factor-density} to see
\begin{align*} \int_{D_1}p_{k}(\bx,\by,U)\,d\bx\,d\by &= \int_{D_1} (1 + o(1)) \prod_{j= 1}^k p_1(x_j,y_j,U)\,d\bx\,d\by + O(e^{-(\log n)^2}). \\
 &=  (1+o(1))\int_{\T_0^{2k}}  \prod_{j= 1}^k p_1(x_j,y_j,U)\,d\bx\,d\by - (1+o(1))\int_{D_0}  \prod_{j= 1}^k p_1(x_j,y_j,U)\,d\bx\,d\by  \\
 &= (1+o(1))\prod_{j= 1}^k \int_{\T_0^2} p_1(x_j,y_j,U)\,dx_j\,dy_j + o(1), \end{align*}
where we have applied Lemma~\ref{th:density-bound} again to see that $p_1(x_j,y_j,U) = O(n^2)$ and used that 
\[ |D_0| = O((\log n)^{4k}n^{-(2k+1/2)}).\]
We may now apply Lemma~\ref{lem:eval-mean} to see that
\[ \prod_{j= 1}^k \int p_1(x_j,y_j,U)\,d\bx\,d\by = (1+o(1))\left(\frac{|U|}{12}\right)^k. \]
Thus for all bounded open sets $U$ and all $k$, we have
\[ \E[(\mu_f(U))_{k}]\ = (1+o(1))\left(\frac{|U|}{12}\right)^k \]
and so we apply Lemma~\ref{lem:method-of-moments} to see that $\mu_f$ tends to a Poisson point processes in the vague topology. 

In order to show that $\nu_n$ converges, it is enough to show that the finite dimensional distributions converge \cite[Theorem 4.2(iii)]{kallenberg}.  For any finite set of intervals, $\mu_n$ and $\nu_n$ match on all of them with high probability.  Since $\mu_n$ converges to the Poisson process of rate $1/12$, its marginal distributions converge to the corresponding marginals of the Poisson process and thus those of $\nu_n$ do as well.

\end{proof}

\begin{proof}[Proof of Corollary \ref{cor:main}]
Let $X_n = \min_{\zeta_k} ||\zeta_k| - 1| n^2$.  Then for each $t \geq 0$, Theorem \ref{thm:main1} implies 
\[\P(X_n \geq  t) \to e^{ - 2\cdot t /12} = e^{-t/6}\,,\]
which is sufficient to conclude that $X_n$ tends in distribution to an exponential random variable with mean $6$.
\end{proof}

\section{The numerator}
\noindent 
We now turn to the first of several sections where we take the reader through a proof of Lemma~\ref{th:density-bound}. In this section, we tackle the numerator of $p_{k}(\bx,\by) = p_{n,k}(\bx,\by)$, as defined at \eqref{eq:def-pk}. That is, write
\[ p_{k}(\bx,\by) = \frac{\alpha_{k}(\bx,\by)}{(2\pi)^{k} |\Sigma|^{1/2}},\] where we have defined $\alpha_{k}(\bx,\by)$ to be 
\begin{equation}\label{eq:def-alpha}  \E\bigg[ \left|X'(x_1)\cdots X'(x_k) Y'(y_1)\cdots Y'(y_k) \right|\, X(x_i)=Y(y_i) =0 \text{ for all }i \in [k]\bigg]. \end{equation} The purpose of this section is to prove the following upper bound on $\alpha_{k}$.

\begin{lemma}\label{cor:num-UB} For $k \in \N$, there is a constant $C_{k} >0$ so that for all 
$\bx = (x_1,\ldots,x_k) \in \T_0^k$ and $\by = (y_1,\ldots,y_k) \in \T_0^{k}$ we have 
\[ \alpha_{k}(\bx,\by) \leq  C_{k}  n^{2k^2 + k}\left(\prod_{i < j} \min\{|x_i - x_j|,n^{-1}  \}^2\cdot \min\{|y_i - y_j|,n^{-1}  \}^2 \right)\,. \]
\end{lemma}

\vspace{2mm}

\noindent Our first move in the direction of Lemma~\ref{cor:num-UB} is a basic property of Gaussian random variables. 

\begin{lemma} \label{lem:expectation-variance}
	Let $(Z_1,\ldots,Z_d)$ be a mean-zero multivariate Gaussian random variable  Then \begin{equation*}
	\E[|Z_1|\cdots |Z_d|] = \Theta_d\left(\sqrt{\Var(Z_1)\cdots \Var(Z_d)} \right) \,.
	\end{equation*}
\end{lemma}
\begin{proof}
	Define the vector $(W_j)_{j \in [d]}$ by setting $W_j = Z_j/\sqrt{\Var(Z_j)}$ and write $$\E[|Z_1|\cdots|Z_d|] = \sqrt{\Var(Z_1)\cdots \Var(Z_d)}\cdot \E[|W_1|\cdots|W_d|]\,.$$
	We now show $\E[|W_1|\cdots|W_d|] = \Theta(1)$.  Let $M = \Cov ((W_j)_{j \in [j]})$ and note that $M_{jj} = 1$ for all $j$.  The map $M \mapsto \E[|W_1|\cdots|W_d|]$ is continuous and takes only positive values. Since the set of covariance matrices with $1$'s on the diagonal is compact, it follows that $\E[|W_1|\cdots|W_d|] $ is bounded above and away from $0$, as desired.\end{proof}

\vspace{4mm}

\noindent With Lemma \ref{lem:expectation-variance} in tow, we only need to understand the conditional variances $$\Var(X'(x_j) \,|\, X(x_i) = Y(y_i) = 0 \text{ for all } i) $$
and similarly for $Y'(y_j)$.  To get a handle on the typical size of $X'(x_j)$ conditioned on $X(x_j)  = Y(y_j) = 0$, we require a consequence of the mean value theorem.

\begin{lemma}\label{lem:numerator-MVT}
	Let $x_1,\ldots,x_k$ be distinct points and let $f$ be a smooth function.  If $f(x_j) = 0$ for all $j \in [d]$ then there exists a $\xi \in [\min_j\{x_j\},\max_j\{x_j\}]$ so that $$f'(x_1) = \frac{f^{(k)}(\xi)}{k!}\prod_{j = 2}^k(x_1 - x_j)\,.$$
\end{lemma}
\begin{proof}
	Consider the polynomial $p(x) = c(x - x_1)(x-x_2)\cdots (x-x_k)$
	with $c = \frac{f'(x_1)}{(x_1 - x_2)\cdots(x_1 - x_k)}\,.$ 
	Note that $p(x_j) = 0$ for all $j \in [d]$ and $p'(x_1) = f'(x_1)$. Now, define $g(x) = f(x) - p(x)$ and note that $g(x_j) = 0$ for all $j \in [d]$ and $g'(x_1) = 0$.  Let $y_1 < \cdots < y_k$ be the points $x_1,\ldots,x_k$ in order; by Rolle's theorem, there are points $\xi_i \in (y_i,y_{i+1})$ for $i \in [d-1]$ with $g'(\xi_i) = 0$.  Since $g'$ is zero on the $k$ distinct points $\{\xi_i\}_{i \in [d-1]}$ and $x_1$, there must exist a $\xi \in [\min_j\{x_j\},\max_j\{x_j\}]$ so that $g^{(k)}(\xi) = 0$. Thus $$0 = g^{(k)}(\xi) = f^{(k)}(\xi) - p^{(k)}(\xi) = f^{(k)}(\xi) - \frac{k! f'(x_1)}{(x_1 - x_2)\cdots (x_1 - x_k)}\,.$$
	Solving for $f'(x_1)$ completes the proof.
\end{proof}

\vspace{4mm}

\noindent We shall also need the following standard fact about Gaussian trigonometric polynomials. We state the next two lemmas for the function $X$, but the same is true of $Y$.

\begin{lemma}\label{lem:sup-bound} For $k \in \Z_{\geq 0}$, let $I \subseteq \T$ be an interval with $|I| = An^{-1}$. Then 
\begin{equation} \label{eq:sup-bound-1}\E \sup_{y \in I} |X^{(k)}(y)| = O_{A,j}(n^{k+1/2}) \end{equation}
and 
\begin{equation} \label{eq:sup-bound-2}\E \left[ \sup_{y \in I} |X^{(k)}(y)|\, \Big\vert\, X(y) = 0 \right]= O_{A,j}(n^{k+1/2}). \end{equation}
\end{lemma}

\begin{proof}
For $x,x_0 \in I$, we may apply Taylor's theorem about $x_0$ to obtain
\begin{equation} \label{eq:Xk-talyor}
	\left|X^{(k)}(x)\right| = \left|\sum_{j \geq 0} \frac{X^{(j+k)}(x_0)}{j!}(x-x_0)^j \right| \leq \sum_{j \geq 0} \frac{|X^{(j+k)}(x_0)|}{j!} A^jn^{-j}\,.
	\end{equation}
To prove \eqref{eq:sup-bound-1} we simply take expectations of both sides and use that 
\[ \E |X^{(j+k)}(x_0)| \leq (\Var (X^{(j+k)}(x_0)) )^{1/2} = O(n^{k+j+1/2}). \]
The proof of \eqref{eq:sup-bound-2} only requires one further twist: we take expectations of both sides of \eqref{eq:Xk-talyor}, while conditioning on on $X(y) = 0$. We then use the property of Gaussian random variables 
\[ \E\left[ |X^{(j+k)}(x_0)| \,|\, X(x_1) = 0 \right] \leq \E |X^{(j+k)}(x_0)| \]
and then proceed as in the proof of \eqref{eq:sup-bound-1}.
\end{proof}

\begin{lemma}\label{lem:num-close}
	For $k \in \N$, let $x_1,\ldots, x_k \in \T$ be points with $|x_1 - x_j| < n^{-1}$ for all $j$.  Then 
 $$\E[|X'(x_1)|\,|\,X(x_1) = X(x_2) = \cdots = X(x_k) = 0  ] \leq C_k n^{k+1/2} \prod_{j = 2}^k |x_1 - x_j| \,, $$
	where $C_k>0$ is a constant depending only on $k$.
\end{lemma}
\begin{proof} Let $a = \min_i x_i$ and $b = \max_i x_i$. Given $X(x_1) = X(x_2) = \cdots = X(x_k) = 0$, apply Lemma \ref{lem:numerator-MVT} to obtain 
$$|X'(x_1)| \leq \max_{\xi \in [a, b]} |X^{(k)}(\xi)| \prod_{j = 2}^k |x_j - x_1|\,.$$
	Taking expectations, conditional on $X(x_1) = X(x_2) = \cdots = X(x_k) = 0$, and using Lemma \ref{lem:sup-bound} completes the proof.\end{proof}

\vspace{4mm}

\begin{proof}[Proof of Lemma~\ref{cor:num-UB}]
	The multidimensional Gaussian random variable 
	\[ (X'(x_1), \ldots, X'(x_k), Y'(y_1),\ldots, Y'(y_k))\] is still Gaussian after we condition on 
	$ X(x_i) = Y(y_i) = 0$, for $i \in [k]$. Thus, we may apply Lemma \ref{lem:expectation-variance} to these conditioned random variables to learn 
	\begin{equation}\label{eq:alphaUB}
	 \alpha_k(\bx,\by) = \E\left[ \left|X'(x_1) \cdots X'(x_k), Y'(y_1)\cdots Y'(y_k)\right| \bigg\vert\,  X'(x_j) = Y'(y_j) = 0 \text{ for all } j\right] \end{equation} is at most
	\begin{equation}\label{eq:alphaUB2} c_k \prod_{i=1}^k \left(\E\left[ | X'(x_i)|^2 \vert\,  X'(x_j) = Y'(y_j) = 0 \text{ for all } j \right] \E\,\left[ |Y'(y_i)|^2\vert\, X'(x_j) = Y'(y_j) = 0 \text{ for all } j \right]\right)^{1/2} . \end{equation}
Define $S_i = \{ j : |x_i - x_j| < n^{-1} \}$. We now use two further properties of Gaussian random variables, \eqref{item:gaussCondition} and \eqref{item:gaussMoment} from Fact~\ref{fact:gaussian}, to bound
\begin{align*}
	\left(\E\left[ | X'(x_i)|^2 \vert\, X'(x_j) = Y'(y_j) = 0 \text{ for all } j \right] \right)^{1/2}
	&\leq \left(\E\left[ | X'(x_i)|^2 \vert\, X'(x_j) = 0 \text{ for all } j \in S_i \right] \right)^{1/2} \\
    &\leq C\E\left[ | X'(x_i)| \vert\, X'(x_j) = 0 \text{ for all } j \in S_i \right]. \end{align*} 
	Thus we can use this together with Lemma~\ref{lem:num-close} to see 
\begin{equation} \label{eq:alpha-termUB} \left(\E\left[ | X'(x_i)|^2 \vert\,  X'(x_j) = Y'(y_j) = 0 \text{ for all } j  \right]\right)^{1/2}  \leq C_k n^{k+1/2} \prod_j\min\{ |x_j-x_i|, n^{-1} \}.\end{equation}
Thus using \eqref{eq:alpha-termUB} in \eqref{eq:alphaUB2}, gives us
\[ \alpha_k(\bx,\by)\leq C_k n^{2k^2+k} \prod_{i < j}\min\{ |x_j-x_i|, n^{-1} \}^2 \cdot \min\{ |y_j-y_i|, n^{-1} \}^2,  \]
as desired.\end{proof}

\section{Understanding the covariance structure} \label{sec:understanding}

In this section we prepare for Section~\ref{sec:det-lower-bound}, where we will obtain a lower bound on the determinant of the covariance matrix $\Sigma_{k}(\bx,\by)$. In that section, we will relate such covariance matrices to 
covariance matrices involving $X,Y$ and their derivatives at a ``well separated'' set of points. In this section, we study the structure of such covariance  matrices, involving $X,Y$ and their derivatives.
Our main goal here will be to prove Lemma~\ref{lem:correlation-LB}. For this, we use the notation 
\begin{equation} \label{eq:X0Y0} X^{(j)}_0(x) := \frac{X^{(j)}(x)}{n^{j + 1/2}} \qquad Y_0^{(j)}(x) := \frac{Y^{(j)}(x)}{n^{j + 1/2}}, \end{equation}
to denote the re-normalized versions of the derivatives of the trigonometric sums $X,Y$ and, for a matrix $A$, we let $\l_{\min}(A)$
denote the smallest eigenvalue of $A$. We will also abuse notation slightly and let $[s,t]$ denote the interval of \emph{integers}  $\{s,s+1,\ldots, t\}$,
when it is clear from context.

\begin{lemma}\label{lem:correlation-LB}
	For $\eps > 0$ and integers $r \geq 1, s \geq 0$ there exists $n(\eps,r,s) \in \N$ so that the following holds. 
If $z_1,\ldots,z_r \in \T_0$ satisfy $d(z_i,z_j) \geq \eps/n$, for all $i \not= j$ and
\begin{equation} \label{eq:corr-LB} 
	\Sigma = \Cov \left(  X^{(j)}_0(z_i), Y_0^{(j)}(z_i) \right)_{ i \in [r],  j\in [0,s]} \end{equation}
then 
\[ \l_{\min}(\Sigma) = \Omega_{\eps,r,s}(1), \]for all $n\geq n(\eps, r,s)$.
\end{lemma}

We approach Lemma~\ref{lem:correlation-LB} by first considering a limit object that captures the covariance structure of $X,Y$, and associated derivatives, at the scale $n^{-1}$. To this end, we define the Gaussian processes $(Z(t),W(t))$, where  $W(t)$ is defined to be the stationary mean-zero Gaussian process with covariance function 
\begin{equation}\label{eq:covWW} \Cov(W(t),W(s)) = \frac{\sin(t-s)}{2(t - s)} = \frac{1}{2}\int_0^1 \cos((t-s)\theta)\,d\theta \end{equation}
and $Z(t)$ is defined to be the stationary mean-zero Gaussian process with the same covariance as $W(t)$ and  
\begin{equation}\label{eq:covWV} \E[Z(t)W(s)] = \frac{1 - \cos(t -s)}{2(t - s)}  = \frac{1}{2}\int_0^1 \sin((t-s)\t) \, d\t.\end{equation}
Our aim in this section will be to prove the following about the process $(Z(t),W(t))$.

\begin{lemma}\label{lem:min-eigenval}
	For $\eps > 0$ and integers $r \geq 1$, $s \geq 0$ there exists $n(\eps,r,s) \in \N$ so that the following holds. If $z_1,\ldots,z_r \in \R$ satisfy $|z_i-z_j| \geq \eps$ and 
	\[ \Sigma_{\infty} :=  \Cov\left(  W^{(j)}(z_i), Z^{(j)}(z_i) \right)_{  i \in [r] j\in [s] }\] then 
	\begin{equation} \label{eq:limiting-det-LB} \l_{\min}(\Sigma_{\infty})   = \Omega_{\eps,r,s}(1) \end{equation}
	for all $n \geq n(\eps,r,s)$.
\end{lemma}
	
\noindent We shall then deduce Lemma~\ref{lem:correlation-LB} by showing that $\l_{\min}(\Sigma) \approx \l_{\min}(\Sigma_{\infty})$, for sufficiently large $n$.

\subsection{The process $(Z(t),W(t))$}
We should remark that it is not actually clear, at this point, that the process $(W(t),Z(t))$ actually \emph{exists}.  
While this is not strictly necessary for our work here, we do pause to take care of this point. Indeed, the question of existence is settled by the following lemma, along with some general machinery.

\begin{lemma}\label{lem:cov-pos-def}
	Let $z_1,\ldots,z_r $ be distinct elements of $\R$.  Then the covariance matrix  \[\Sigma_{\infty} := \Cov (Z(z_i), W(z_i))_{i=1}^{r}\] is positive definite.
\end{lemma}

We don't prove lemma~\ref{lem:cov-pos-def} here as it will follow from our more general Lemma~\ref{lem:vSv-int-form}. But with Lemma~\ref{lem:cov-pos-def} in hand, we may now apply Kolmogorov's extension theorem \cite[Section 1.2]{adler2009random} to learn that $(Z(t),W(s))$ exists.  Moreover, since the function $\sin(x)/x$ is smooth, we may assume that $(Z(t),W(s))$ has smooth paths \cite[page 30]{azais2009level}.

\vspace{4mm}
We now turn to explore the covariance structure of the derivatives of this process. By differentiating under the integral we have
\begin{equation}\label{eq:covWaWb} \Cov(W^{(a)}(t),W^{(b)}(s)) = \frac{(-1)^b}{2}\int_{0}^1 \t^{a+b}\cos^{(a+b)}((t-s)\theta)\,d\theta\,. \end{equation}
and 
\begin{equation}\label{eq:covWaZb} \Cov(W^{(a)}(t),Z^{(b)}(s)) = \frac{(-1)^b}{2}\int_0^1 \t^{a+b} \sin^{(a+b)}((t-s)\t)\, d\t,\end{equation}
where $\cos^{(k)}, \sin^{(k)}$ denote the $k$th derivative of cosine and sine, respectively. 
In order to work with covariance matrices of $Z,W$ and their derivatives, we prove the following lemma that allows us to express $v^T\Sigma v$ in a convenient form. We remark that this useful form appears in \cite[Lemma 5]{azais-universility}, but only for the process $W$ and its derivative.

\begin{lemma}\label{lem:vSv-int-form} For $\eps > 0$ and integers $r \geq 1$, $s \geq 0$ and $z_1,\ldots,z_r \in \R$, let 
	\[ \Sigma :=  \Cov\left(  W^{(j)}(z_i), Z^{(j)}(z_i) \right)_{i \in [r], j\in [0,s]}.\]
Then 
\[ v^t\Sigma v = \frac{1}{2}\int_0^1 |F_v(\t)|^2 \, dt, \] where 
\[ F_v(\t) := \sum_{a,b} (x_{a,b} + i y_{a,b})(i\t)^be^{z_a i\t},\]
$v = (x,y)$, $x = (x_{i,j})_{i \in [r],j\in [0,s]}$, and $y = (y_{i,j})_{i\in [r],j\in [0,s]}$.
\end{lemma}
\begin{proof}
To understand the indexing of this vector $v$, we note that we can think of the rows and columns of $\Sigma$ as indexed by 
$Z^{(j)}(z_i)$, for all $i \in [r],j \in [0,s]$ along with $W^{(j)}(z_i)$ for all $i \in [r],j \in [0,s]$. 
So we may write $v = (x,y)$ where $x = (x_{i,j})_{i\in [r],j\in [s]}$ and $y = (y_{i,j})_{i\in [r],j\in [s]}$ and  
expand $v^T\Sigma v$ as 
\begin{align} \label{eq:vSv-exp}
v^T \Sigma v &= \sum_{a,b,c,d} x_{a,b}x_{c,d} \Cov(W^{(b)}(z_a)W^{(d)}(z_{c})) + \sum_{a,b,c,d} x_{a,b}y_{c,d} \Cov(W^{(b)}(z_a)Z^{(d)}(z_{c}))  \\
&\qquad+\sum_{a,b,c,d} y_{a,b}x_{c,d} \Cov(Z^{(b)}(z_a)W^{(d)}(z_{c})) + \sum_{a,b,c,d} y_{a,c}y_{c,d} \Cov(Z^{(b)}(z_a)Z^{(d)}(z_{c})). \nonumber  \end{align}
Working with the first of the terms on the right-hand-side, we have 
\[ \sum_{a,b,c,d} x_{a,b}x_{c,d} \Cov(W^{(b)}(z_a)W^{(d)}(z_{c})) =  \frac{1}{4}\sum_{a,b,c,d} x_{a,b}x_{c,d}\int_0^1\left( (i \theta)^{b}(-i\theta)^{d} e^{i(z_a - z_c)\theta} + (i \theta)^{d}(-i\theta)^{b}e^{i(z_c - z_a)\theta}\right)\,d\theta \]
\[ =\frac{1}{2}\sum_{a,b,c,d} x_{a,b}x_{c,d}\int_0^1 (i \theta)^{b}(-i\theta)^{d} e^{i(z_a - z_c)\theta} \,d\theta = \frac{1}{2} \int_0^1 \left|\sum_{a,b} x_{a,b} (i\t)^b e^{i z_a \theta} \right|^2\,d\theta\]

where the second equality follows from swapping the roles of the pairs $(a,b)$ and $(c,d)$ and combining the two sums.
Similarly, simplifying the other three terms in \eqref{eq:vSv-exp} using  \eqref{eq:covWaWb}, \eqref{eq:covWaZb}, allows us to rewrite \eqref{eq:vSv-exp} to obtain
\[ v^T\Sigma v =  \frac{1}{2}\int_0^1\left| F_v(\theta)  \right|^2 \, d\t,\] as desired. \end{proof}

\vspace{4mm}

We now prove a general lower bound for the exponential-type polynomials that appear in Lemma~\ref{lem:vSv-int-form}. 
\begin{lemma}\label{lem:triglowbound} For $r \geq 1 , s \geq 0, \eps >0 $, let $z_1,\ldots,z_r \in \R$ satisfy $|z_i - z_j| \geq \eps$ and let
\[ F(\t) = \sum_{i =1}^r\sum_{j=0}^s \alpha_{i,j}\t^je^{iz_i\t}.\]
Then
\begin{equation} \label{eq:triglowbound}\int_0^1|F(\t)|^2 \, d\t \geq c_{r,s,\eps}  \sum_{i =1}^r\sum_{j=0}^s |\alpha_{i,j}|^2. \end{equation}
\end{lemma}
\begin{proof}
We apply induction on $r$. The statement is clear for $r = 1$. For the induction step at $r$, we define
\[ c^{\ast}_r = \min_{t < r} c_{t,s,\eps}, \]
and we will choose $M$ to be sufficiently large compared to $1/c^{\ast}_r$ $r$ and $s$. We consider two cases; (case 1) all of the $z_i$ are contained in an interval of length $rM$ \emph{or}  (case 2) there exists a partition into non-empty sets $A \cup B = [r]$
so that $|z_i-z_{i'}| > M$ for all $i \in A$ and $i'\in B$. We start by dealing with this latter case: let us write $F = F_A + F_B$ where 
$F_A := \sum_{(i,j): i\in A }  \alpha_{i,j}\t^je^{iz_j\t}$ and $F_B$ is defined similarly We have
\begin{equation}\label{eq:intstuff} \int_0^1 |F|^2 \, d\t = \int_0^1 |F_A|^2 \, d\t + \int_0^1 |F_B|^2 \, d\t 
+ \sum_{i\in A, i' \in B} 2\alpha_{i,j}\alpha_{i',j'}\int_0^{1} \t^{j+j'} \cos((z_i-z_j)\t)\, d\t, \end{equation}
where this last sum is over all pairs $(i,j),(i',j')$, where $i \in A $ and $i' \in B$.
Now since $|z_i-z_{i'}|>M $ for all such pairs, we can integrate by parts to see 
 \[ \left| \int_0^{1} \t^{j+j'} \cos((z_i-z_j)\t)\right| \leq 4r/M .\]
Using this, along with $2|xy| \leq x^2+ y^2$, we see
\[ \left| \sum_{i\in A, i' \in B} 2\alpha_{i,j}\alpha_{i',j'}\int_0^{1} \t^{j+j'} \cos((z_i-z_j)\t),\right| \leq \frac{8r}{M} \sum_{i\in A,i'\in B} |\alpha_{i,j} \alpha_{i',j'}| \leq \frac{4r^2s}{M} \sum_{i,j} |\alpha_{i,j}|^2. \]
On the other hand, we my apply induction to the first two terms on the right hand side of \eqref{eq:intstuff} to obtain
\[ \int_0^1 |F|^2 \, d\t \geq c^{\ast}_r \sum_{i,j} |\alpha_{i,j}|^2 - (4r^2s/M) \sum_{i,j} |\alpha_{i,j}|^2. \]
Thus we may choose $M$ to be sufficiently large in terms of $c^{\ast}_r$, $r$ and $s$ so that the above sum is at least $c_r^{\ast}/2\sum_{i,j} |\alpha_{i,j}|^2$,
as desired.

In the case that $z_1,\ldots,z_r$ lie in an interval of length $rM$, we apply a compactness argument. First note that we may assume that $z_1,\ldots,z_r \in [0,rM]$, as replacing $\{z_1,\ldots, z_r\}$ with $\{ z_1 + T,\ldots z_r + T\}$ does not change the value of the integral. Also note that we may assume that $\sum_{i,j} |\alpha_{i,j}|^2=1$, by scaling both sides of \eqref{eq:triglowbound}.
Now, define the function 
\[ I(z,\alpha) := \int_0^1 |F(\t)|^2\, d\t,\]
where $z = (z_1,\ldots,z_r)$ and $\alpha = (\alpha_{i,j})_{i,j}$ and note that $I(z,\alpha)$ is a continuous function of its variables and that $I$ is always positive when $(\alpha_{i,j})_{i,j}$ is non-zero and $z_1,\ldots,z_r$ are distinct.
Now define the compact set
\[ S_{r,\eps,M} := \{ z \in [0,rM]^r : \eps \leq  |z_i-z_j|\,\,  \forall i\not= j \} \times \{ \alpha \in \C^{r\times (s+1)} : |\alpha|_2 = 1\}. \]
Since $I$ is always positive on $S_{r,\eps,M}$ it attains a positive minimum value on $S_{r,\eps,M}$, this concludes the proof. \end{proof}

\vspace{3mm}

We now are in a position to prove Lemma~\ref{lem:min-eigenval}, our version of Lemma~\ref{lem:correlation-LB} for the limiting process $(W(t),Z(t))_t$.

\begin{proof}[Proof of Lemma~\ref{lem:min-eigenval}]
We use the variational definition of $\l_{\min}(\Sigma_{\infty})$ along with Lemma~\ref{lem:vSv-int-form} to write
\[ \l_{\min}(\Sigma_{\infty}) = \min_{v: |v|_2 = 1} v^T\Sigma_{\infty} v =(1/2)\min_{v: |v|_2 =1} \int_0^1 |F_v(\t)|^2 \, d\t, \]
where $F_v$ is finite sum of the form $F_v(\t) := \sum_{i,j} \alpha_{i,j} \t^i e^{iz_i\t}$ and the $\alpha_{i,j}$ are complex numbers satisfying 
$\sum_{i,j} |\alpha_{i,j}|^2 = 1$. We may now apply Lemma~\ref{lem:triglowbound} to finish the proof of Lemma~\ref{lem:min-eigenval}.
\end{proof}

\vspace{4mm}

\subsection{Proof of Lemma~\ref{lem:correlation-LB} from Lemma~\ref{lem:min-eigenval}}

We now turn to prove our main lemma of this section by approximating covariance matrices in the statement of Lemma~\ref{lem:correlation-LB} with the covariance matrices of our limiting object $(Z(t),W(t))$.

\begin{lemma}\label{lem:covar-approx0} For $r,s \in \Z_{\geq 0}$, let $x_1,\ldots,x_r \in \T_0$, let 
\[ \Sigma = \Cov \left(  X^{(j)}_0(z_i), Y_0^{(j)}(x_i) \right)_{ i \in [r],  j\in [0,s]}\]
and let
\[ \Sigma_{\infty} :=  \Cov\left(  W^{(j)}(n x_i), Z^{(j)}(nx_i) \right)_{  i \in [r],  j \in [0, s] }.\]
Then 
\[ \|\Sigma - \Sigma_{\infty}\|_{op} = O_{r,s}(n^{-1/2}). \]

\end{lemma}

\noindent We show this by showing that $\Sigma - \Sigma_{\infty}$ tends to the zero matrix entry-wise.

\begin{lemma}\label{lem:covar-approx} For $a,b\in \Z_{\geq 0}$,  we have
\[ \max_{x,y \in \T_0 }\left|\Cov\left(X_0^{(a)}(x), X_0^{(b)}(y) \right) - \Cov\left(W^{(a)}(nx),W^{(b)}(ny)\right)\right| = O(n^{-1/2}); \]
\[ \max_{x,y \in \T_0}\left|\Cov\left(X_0^{(a)}(x), Y_0^{(b)}(y) \right) - \Cov\left(W^{(a)}(nx),Z^{(b)}(ny)\right)\right| = O(n^{-1/2}); \]
\[ \max_{x,y \in \T_0}\left|\Cov\left(Y_0^{(a)}(x), Y_0^{(b)}(y) \right) - \Cov\left(Z^{(a)}(nx),Z^{(b)}(ny)\right)\right| = O(n^{-1/2}). \]
\end{lemma}
\noindent The proof Lemma~\ref{lem:covar-approx} relies on the following basic fact. Define
\[ D_{n,d}(x) := \sum_{k=0}^n k^d\cos xk \qquad S_n(x) := \sum_{k=0}^n k^d\sin xk \]
for each $d \geq 0$.

\begin{fact}\label{fact:approx-int}  Let $n,d \in\Z_{\geq 0}$. We have
\[\max_{x \in \T} \left| n^{-(d+1)}D_{n,d}(x) - \int_0^1 t^d \cos(xn t)  \, dt \right|  = O(n^{-1/2}); \]
\[\max_{x \in \T} \left| n^{-(d+1)}S_{n,d}(x) - \int_0^1 t^d \sin(xn t)  \, dt \right|  = O(n^{-1/2}). \]
\end{fact}

\begin{proof}
We prove only the first part of the fact and note the other is almost identical. We write
\[ n^{-(d+1)}D_{n,d}(x) = n^{-1}\sum_{k=0}^n (k/n)^d\cos\left( xn (k/n) \right) = n^{-1}\sum_{k=0}^n g(k/n), \]
where we set $g_{n,x,d}(t) := t^{d}\cos(xn t)$. Now note that 
\begin{equation}\label{eq:g'} |g'_{n,x,d}(t)| = |dt^{d-1}\cos(xn t) + xnt^d\sin(xnt)| \leq 2dxn, \end{equation}
for $t \in [0,1]$. Thus, using \eqref{eq:g'} and an effective form of convergence of the Riemann integral, we have 
\[ \left| n^{-(d+1)}D_{n,d}(x) - \int_0^1 g(t) \, dt \right| =  \left| n^{-1}\sum_{k=0}^n g(k/n)  - \int_0^1 g(t) \, dt \right| \leq n^{-1}\int_0^1|g'(t)| \, dt = O(x). \]
So this proves Fact~\ref{fact:approx-int} when $x < n^{-1/2}$. For $x > n^{-1/2}$ we show that both the sum and integral are small. Starting with the integral, integrate by parts to see 
$$\int_0^1 t^d \cos(xnt)\,dt = \frac{\sin(xn)}{xn} - \frac{d}{xn} \int_0^1 t^{d-1}  \sin(xnt)\,dt = O(1/(xn)) = O(n^{-1/2})\,.$$ 
On the other hand, apply Abel's summation formula to express
\[ D_{n,d}(x) = n^dD_{n,1}(x) + d \int_0^n t^{d-1} D_{\lfloor t \rfloor,1 } (x) \, dt, \]
and then use the bound $|D_{m,1}(x)| \leq 1/x$ for all $m \in \N$. \end{proof}

\vspace{4mm}

\begin{proof}[Proof of Lemma~\ref{lem:covar-approx}] 
We treat the first case and note that the others are similar. We have  
\[ \Cov\left(X(x), X(y) \right) = \sum_{k=0}^n \cos(xk)\cos(yk) = (1/2)D_n(x-y) + (1/2)D_n(x+y). \] 
Thus applying the differential operator $\left(\frac{d}{dx}\right)^a \left(\frac{d}{dy}\right)^b$ to both sides and multiplying by $n^{-(a+b+1)}$ gives
\begin{equation}\label{eq:CovasDD}  \Cov\left(X_0^{(a)}(x), X_0^{(b)}(y) \right) = (1/2)(-1)^b n^{-(a+b+1)} D^{(a+b)}_n(x-y) + (1/2)n^{-(a+b+1)}D_n^{(a+b)}(x+y). \end{equation}
To deal with the second term on the right hand side of \eqref{eq:CovasDD}, we see that $d(x+y,2\pi\Z) > n^{-1/2}$, since $x,y \in \T_0$, and therefore we have $(1/2)n^{-(a+b+1)}D_n^{(a+b)}(x+y) = O(n^{-1/2})$, where the bound is uniform over all $x,y \in \T_0$.
We can then apply Fact~\ref{fact:approx-int} to \eqref{eq:CovasDD} to conclude  
\[\Cov\left(X_0^{(a)}(x), X_0^{(b)}(y) \right) = (-1)^b/2\int_{0}^1 \t^{a+b}\cos^{(a+b)}((t-s)\theta)\,d\t + O(n^{-1/2}), \]
which is $\Cov\left(W^{(a)}(nx),W^{(b)}(ny)\right)$, by definition. \end{proof}
\begin{proof}[Proof of Lemma~\ref{lem:correlation-LB}]
We compare the matrices
\[ \Sigma := \Cov\left( X^{(j)}_0(x_i), Y_0^{(j)}(x_i) \right)_{ i \in [k],  j \in [0,s] } \qquad \Sigma_{\infty} := \Cov\left( W^{(j)}(n x_i ), Z^{(j)}(n x_i)  \right)_{i \in [k],  j \in [0,s]} . \]
Using the variational definition of the least eigenvalue and the semidefinite property of $\Sigma_{\infty},\Sigma$, we write
\[ |\l_{\min}(\Sigma) -  \l_{\min}(\Sigma_{\infty}) |
= \left|\min_{v : |v|_2 = 1} v^T\Sigma v  - \min_{v : |v|_2 = 1} v^T \Sigma_{\infty} v \right| 
\leq \left|\min_{v : |v|_2 = 1}\left\lbrace v^T\Sigma v  -  v^T \Sigma_{\infty} v \right\rbrace \right|, \]
which, by Lemma~\ref{lem:covar-approx0}, is at most 
\[ \|\Sigma - \Sigma_{\infty}\|_{op} = o(1). \]
Since $\l_{\min}(\Sigma_{\infty}) = \Omega_{r,s,\eps}(1)$, by Lemma~\ref{lem:min-eigenval}, (for $n$ large compared to $\eps,r,s$) we conclude that 
\[ \l_{\min}(\Sigma) = \Omega_{r,s,\eps}(1), \] 
for $n$ sufficiently large, compared to $r,s,\eps$.
\end{proof}

\section{The Determinant of the covariance matrix} \label{sec:det-lower-bound}

In this section we supply a complementary result to Lemma~\ref{cor:num-UB} by proving a lower-bound for the denominator in the expression \eqref{eq:def-pk}. Here we write 
\[ \Sigma_{k}(\bx,\by) := \Cov\left(X(x_i),Y(y_i)\right)_{i\in [k]},  \]
where $\bx = (x_1,\ldots,x_k)$ and $\by = (y_1,\ldots,y_k)$.

\begin{lemma}\label{lem:det-LB-general}
	For $k \in \N$ there exists $c_{k} >0$  and $n(k) \in \N$, so that for all $\bx = (x_1, \ldots, x_k) \in \T_0^k$ and $\by = (y_1, \ldots , y_k) \in \T_0^k$ 
we have 	
	 \begin{equation*}
	\det \Sigma_{k}(\bx,\by)  \geq c_{k}  n^{2k^2} \left(\prod_{i < j} \min\{ |x_j - x_i| ,1/n\}^2 \cdot \min\{ |y_j - y_i| ,1/n\}^2\right), \end{equation*} for all $n > n(k)$.
\end{lemma}

To prove this, we will cluster the points $\{x_i,y_j\}_{i,j}$ into groups so that points are ``near'' to points in their own group and sufficiently ``far'' from points in the other groups. In the following subsection, we make a few preparations for working within these clusters of ``near'' points.


\subsection{Near points and derivatives}
Assume that 
$\{x_1,\ldots,x_k\} \cup \{ y_1,\ldots,y_{\ell}\}$ are clustered tightly around points $z_1,\ldots,z_k$, which, themselves, are reasonably well separated.
In this subsection we show that we can compare the covariance matrix of $\Sigma_{k}(\bx,\by)$ (where $\bx,\by$ represent these ``clustered'' set of points) with a covariance matrix of $X^{(j)},Y^{(j)}$ evaluated at the points
$z_1,\ldots,z_r$. This will then allow us to use Lemma~\ref{lem:correlation-LB}, the main result of Section~\ref{sec:understanding}, to prove Lemma~\ref{lem:det-LB-general}.

A key technical device here will be a mean-value theorem for, so called, \emph{divided differences}.  Given $x_0  < x_1< \cdots < x_k $ set 
$x = (x_0,x_1,\ldots,x_k)$ and define \emph{the divided differences with respect to} $x$, for a sequence $y_0,\ldots,y_k$, inductively by $[y_i]_x := y_i$, for each $i$, and
\begin{equation}\label{eq:divdifdef}
 [y_i,\ldots,y_{i + j}]_x:= \frac{[y_{i+1},\ldots,y_{i+j}]_x-  [y_i,\ldots,y_{i+j-1}]_x}{x_{i +j} - x_i}\,.
\end{equation}
 For a function $f$, we further extend this definition, by defining \[f[x_0,\ldots,x_k]_x := [f(x_0),\ldots,f(x_k)]_x .\]

The reason for this definition becomes apparent with the following version of the mean value theorem, which is often attributed to Schwarz\footnote{For more information on divided differences and a proof of Lemma~\ref{lem:MVT-divided-diff}, see the survey \cite{deBoor}.}.

\begin{lemma}\label{lem:MVT-divided-diff}
	Let $x =(x_0,\ldots,x_k) \in \R^{k+1}$ satisfy $x_0 < x_1 < \cdots < x_k$ and let $f$ be a smooth function. Then there exists $\xi \in [\min_j x_j,\max_j x_j]$ so that $$ f[x_0,\ldots,x_k]_x = \frac{f^{(k)}(\xi) }{k!}\,.$$
\end{lemma}

\noindent For our application of Lemma~\ref{lem:MVT-divided-diff}, we will need to use some basic properties of the linear map
 $\Delta_{x} :\R^{k+1} \to \R^{k+1}$ defined by 
$$\Delta_{x} y  = \left([y_0]_x,[y_0,y_1]_x,\ldots,[y_0,y_1,\ldots,y_{k+1}]_x  \right), $$
for all $y = (y_0,\ldots,y_{k+1}) \in \R^{k+1}$.
\begin{lemma}\label{lem:det-divided-diff}
Let $x = (x_1,\ldots,x_k) \in \R^{k}$ where $x_1 < \cdots < x_k$ and let $\Delta_{x}$ be as above. Then $\Delta_{x}$ is a linear map with
\[ \det(\Delta_{x}) = \prod_{1\leq i < j\leq n}(x_j - x_i)^{-1}.\]
\end{lemma}
\begin{proof} The linearity of $\Delta_x$ is clear from the definition. To calculate the determinant of $\Delta_x$, we apply induction on the dimension $k$. The basis step holds by definition. Now, let $e_k = (0,\ldots,0, 1) \in \R^k$ denote the standard unit vector and note that 
\[ \Delta_x e_k = [y_1,y_2,\ldots,y_k]_x e_k.\]
To calculate $[y_1,\ldots,y_k]_x$ we use that $[0,\ldots,0]_x = 0$ to see
\[ [y_1,\ldots,y_k]_x = (x_1 - x_k)^{-1} ([y_1,\ldots,y_{k-1}]_x - [y_2,\ldots,y_k]_x) = (x_k-x_1)^{-1}[y_2,\ldots y_k]_x.  \]
Iterating this gives $[y_1,\ldots,y_k]_x = \prod_{i<k} (x_k-x_i)^{-1}$.
Now note that if $y \perp e_k$ then $\Delta_x y = \Delta_{x'}y'$ where $x',y' \in \R^{k-1}$ are the vectors $x,y$ with the $k$th coordinates removed.
Thus \[ \det(\Delta_{x}) = [y_1,y_2,\ldots,y_k]_x \det(\Delta_{x'}) = \prod_{i < j} (x_j - x_i)^{-1},\] by induction. \end{proof}

\vspace{4mm}

\begin{lemma}\label{lem:finite-diff-and-deriv}
For $\eps >0$, let $x_0,\ldots ,x_k \in \R$ satisfy $|x_0 - x_j| \leq \eps/n$ for all $0\leq j\leq k$.  Then 
$$\Var\left( \frac{X[x_0, \ldots,x_k]}{n^{k+1/2}} - X_0^{(k)}(x_0)/k!  \right) = O(\eps^2)\,.$$
\end{lemma}

\begin{proof}
Since 
\[ E := X[x_0,\ldots,x_k]n^{-k + 1/2} - X_0^{(k)}(x_0)/k! \] 
is a linear combination of values of $X,X^{(k)}$, which are jointly Gaussian, $E$ itself is a Gaussian random variable. 
Thus it is sufficient to show that $\E\, |E| = O(\eps)$. For this, we apply Lemma~\ref{lem:MVT-divided-diff} to obtain
\begin{equation}\label{eq:X[]dd} X[x_0,\ldots,x_k] = \frac{X^{(k)}(\xi)}{k!}  \end{equation}
for some $\xi \in [x_1 -  \eps n^{-1},x_1 + \eps n^{-1}] =: I$ and so applying the (standard) mean value theorem, we have 
\begin{equation}\label{eq:mvt2} X^{(k)}(\xi) - X^{(k)}(x_1)  = (x_1-\xi)X^{(k+1)}(\xi'),   \end{equation}
for some $\xi' \in I$. We now want to bound
\begin{equation}\label{eq:E|E|} \E\, |E| = \E\,  \left|X[x_0,x_1,\ldots,x_k] - X^{(k)}(x_0)/k!\right|n^{-(k+1/2)},\end{equation}
from above. Using \eqref{eq:X[]dd} along with \eqref{eq:mvt2}  tells us that \eqref{eq:E|E|} is at most
\[ (\eps n^{-1}) n^{-k + 1/2}\E\, \sup_{y \in I} \left|X^{(k+1)}(\xi)\right| = O(\eps)\,,\]
where the last equality follows form Lemma~\ref{lem:sup-bound}. 
\end{proof}


We apply Lemma~\ref{lem:finite-diff-and-deriv} to arrive at the main result of this subsection.

\begin{lemma} \label{lem:main-compare} For $\delta >0$, let $\{x_{i,j} \}_{i \in [r],j \in [k]}, \{y_{i,j} \}_{i \in [r],j \in [k]}, \{z_i\}_{i=1}^r \subset \T$ 
be such that
\[ |z_i-x_{i,j}| < \delta/n\, \textit{   and   }\, |z_i-y_{i,j}| < \delta/n,  \] for all $i,j$. Let $\Sigma$ be the covariance matrix of 
the joint distribution
\[ \left(\frac{X[x_{i,0}]}{n^{1/2}}, \ldots, \frac{X[x_{i,0},\ldots,x_{i,k}]}{n^{k+1/2}},
 \frac{Y[y_{i,0}]}{n^{1/2}},\ldots,\frac{Y[y_{i,0},\ldots,y_{i,k}]}{n^{k+1/2}} \right)_{i=1}^r\]
 and let $\Sigma'$ be the covariance matrix of the joint distribution  
\[ \left( \frac{X_0^{(0)}(z_i)}{0!},\ldots,\frac{X_0^{(k)}(z_i)}{k!}, 
\frac{Y_0^{(0)}(z_i)}{0!},\ldots, \frac{Y^{(k)}(z_i)}{k!} \right)_{i=1}^r.\]
Then 
\[ \|\Sigma - \Sigma'\|_{op} = O_{r,k}(\delta).\]
\end{lemma}
\begin{proof}

We prove this by showing that all entries in the matrix $\Sigma - \Sigma_0$ are at most $3\delta$ in absolute value. An entry in the matrix $\Sigma - \Sigma_0$ is of the form 
\begin{equation}\label{eq:terminS0-S'} \E \frac{X[x_{0,a},\ldots,x_{i,a}]}{n^{i+1/2}}\frac{Y[y_{0,b},\ldots,y_{j,b}]}{n^{j+1/2}} 
- \E\left(  i!^{-1} X_0^{(i)}(z_a) \cdot j!^{-1} Y_0^{(j)}(z_b) \right) \end{equation}
or like this with the occurrences of $X$ replaced with occurrences of $Y$ or vice versa. However these cases are similar and so we ignore them.
Momentarily suppressing the first subscript, we may express the first term in \eqref{eq:terminS0-S'} as
\[ \E \frac{X[x_0,\ldots,x_i]}{n^{i+1/2}}\frac{Y[y_0,\ldots,y_j]}{n^{j+1/2}} 
= \E \left( X_0^{(i)}(z_a) i!^{-1}  - E_1\right)\left( Y_0^{(j)}(z_b) j!^{-1} - E_2\right), \] 
where $E_1$, is defined as $ n^{-(i+1/2)}X[x_{0,a},\ldots,x_{i,a}] - (i!)^{-1} X_0^{(i)}(z_b)$ and $E_2$ is defined symmetrically. We now expand this out and obtain four terms, the first of which is 
\[ \Cov\left( i!^{-1} X_0^{(i)}(z_a) , j!^{-1}Y_0^{(j)}(z_b) \right)\]
and is the corresponding term in the matrix $\Sigma'$. So to finish we need to bound the three ``error'' terms
\[ -\E E_1 X_0^{(j)}(x_0)j!^{-1} - \E E_2 X_0^{(i)}(x_0) i!^{-1} + \E E_1E_2. \]
We see that each of these are individually at most $\delta$ by applying Cauchy-Schwarz and then Lemma~\ref{lem:finite-diff-and-deriv}.

As a result, we see that $\Sigma_0 - \Sigma'$ is a $2r(k+1) \times 2r(k+1)$ matrix with all entries at most $3\delta$. To finish, we use 
that the operator norm of a matrix is at most its Frobenius norm. So if we set $A = \Sigma_0 - \Sigma$, we have 
$ \|A\|_{op} \leq \|A\|_{F} = \left| \sum_{i,j} A_{i,j}^2\right|^{1/2} = O_{k,r}(\delta)$, as desired.\end{proof}

\vspace{4mm}

\subsection{Proof of Lemma~\ref{lem:det-LB-general}}

We now turn to prove our main lemma of this section, Lemma~\ref{lem:det-LB-general}. At the heart of the proof is Lemma~\ref{lem:det-close}, which allows us to deal with points that are grouped together in groups which are far apart from each other.
As is common, we use the notation $B_\delta(z) \subseteq \R$ to denote the set of points within distance $\delta$ of $z$ and we recall that 
\[ \T_0 = \{ x \in (0,\pi): d(x, \pi \Z) > n^{-1/2} \}. \]
We also need the following elementary fact.
\begin{fact}\label{fact:det-linear-comb}
	If the random vector $\bx$ has covariance matrix $\Sigma$ then the random vector $A \bx$ has covariance matrix $A \Sigma A^T$.
\end{fact}
\begin{lemma}\label{lem:det-close}
For $k, r \in \N$ and $\eps >0$ there exists $\delta = \delta(k,\eps) \in (0,\eps)$, $c_{k,\eps} > 0$ and $n(k,\eps)$ so that the following holds. Let $z_1,\ldots,z_r \in \T_0$
and let $x_1,\ldots,x_k ,y_1,\ldots,y_k \in \T_0$ satisfy
\[ \{x_1,\ldots,x_k,y_1,\ldots,y_{k}\} \subseteq \bigcup_{i=1}^r \bar{B}_{\delta/n}(z_i) \]
and $|z_i - z_j| > \eps/n$.
If we set $\bx =(x_1,\ldots,x_k)$ and $\by = (y_1,\ldots,y_{k})$ then 
\[ \det \Sigma_{k}(\bx,\by) 
	\geq  c_{k,\eps} n^{2k^2} \prod_{i < j} \left(\min\{|x_j - x_i| ,1/n\}^2 \cdot \min\{|y_j - y_i|, 1/n\}^2\right) \,, \]
for all $n \geq n(k,\eps)$.
\end{lemma}
\begin{proof}
We let $\delta \in (0,\eps/2)$, which we shall specify later. Set $\Sigma := \Sigma_{k}(\bx,\by)$ and define
 \[ S_x(z_i) := \{ j : x_j \in B_{\delta/n}(z_i)\}, \qquad S_y(z_i) := \{ j  : y_j \in B_{\delta/n}(z_i)\}.\]
Note that, by choice of $\delta$, these sets are disjoint.

We now define the linear operator $\Delta$ on $\R^{2k}$ that ``differences'' points $x_i,x_{i'}$, in the sense \eqref{eq:divdifdef}, with $i,i' \in S_x(z_j)$, for each part of the partition (and likewise for the $y_i$). For this, split up the coordinates of $\R^{2k}$,
\[ \R^{2k} = \R^{S_x(z_1)}\times \cdots \times \R^{S_x(z_k)} \times \R^{S_y(z_1)} \times \cdots \times \R^{S_y(z_k)},\] according to the partition
$\{S_y(z_i),S_x(z_i)\}_i$, and write $( \bx,\by) = (\bx_1,\ldots,\bx_k,\by_1,\ldots,\by_k)$. For a point
\[ (\bu_1,\ldots,\bu_{k}, \bv_1,\ldots,\bv_k) \in\R^{2k}  = \R^{S_x(z_1)}\times \cdots \times \R^{S_x(z_k)} \times \R^{S_y(z_1)} \times \cdots \times \R^{S_y(z_k)}\] we define $\Delta$ by setting 
\[ \Delta (\bu_1,\ldots,\bu_{k}, \bv_1,\ldots,\bv_k) = (\Delta_{\bx_1} \bu_1, \ldots, \Delta_{\bx_k} \bu_k,\Delta_{\by_1} \bv_1, \ldots, \Delta_{\by_k} \bv_k ).\]
Of course, 
\begin{equation}\label{eq:det-split} \det\Sigma  = (\det \Delta )^{-2}\det( \Delta \Sigma \Delta^T) \end{equation}
and, by Fact~\ref{fact:det-linear-comb}, we see that $ \Delta \Sigma \Delta^T$ is the covariance matrix of the random variable $\Delta (\bx,\by)$.
After renaming the points 
\[ S_x(z_i) = \{x_{i,0},\ldots,x_{i,k(i)} \} \qquad S_y(z_i) = \{ y_{i,0},\ldots,y_{i,k'(i)}\}, \] for each $i \in [r]$, we may use the definition
of $\Delta$ to write
\[ \Delta \Sigma \Delta^T =\Cov\left( X[x_{i,0}],\ldots,X[x_{i,0},\ldots,x_{i,k(i)}],Y[y_{i,0}],\ldots,Y[y_{i,0},\ldots,y_{i,k'(i)}] \right)_{i=1}^r. \]
We now rescale rows and columns corresponding to the terms $X[x_{i,0},\ldots,x_{t,k(i)}]$ of ``depth'' $t$ by $n^{-(t+1/2)}$ to obtain
\begin{equation}\label{eq:det-rescale} \det \Delta \Sigma \Delta^T = n^{T} \det \Cov\left(\frac{X[x_{i,0}]}{n^{1/2}}, \ldots, \frac{X[x_{i,0},\ldots,x_{i,k(i)}]}{n^{(k(i)+1/2)}},
 \frac{Y[y_{i,0}]}{n^{1/2}},\ldots,\frac{Y[y_{i,0},\ldots,y_{i,k'(i)}]}{n^{k'(i)+1/2}} \right)_{i=1}^r, \end{equation}
where we have set 
\[ T := \sum_{i=1}^r (k(i)+1)^2 +  \sum_{i=1}^r (k'(i)+1)^2.\]
Now let $\Sigma_0$ be the rescaled matrix in \eqref{eq:det-rescale}. We compare this renormalized matrix $\Sigma_0$ with the covariance matrix of the (re-scaled) derivatives of $X,Y$ at $\{ z_i \}$. That is, the matrix
\[ \Sigma' := \Cov\left( \frac{X_0^{(0)}(z_i)}{0!},\ldots,\frac{X_0^{(k(i))}(z_i)}{(k(i))!}, 
\frac{Y_0^{(0)}(z_i)}{0!},\ldots, \frac{Y_0^{(k'(i))}(z_i)}{(k'(i))!} \right)_{i=1}^r\,.  \]
We apply Lemma~\ref{lem:correlation-LB} (the main result of Section~\ref{sec:understanding}) to learn
\begin{equation}\label{eq:lmin-lb} \l_{\min}(\Sigma') = \Omega_{k,\eps}(1), \end{equation}
for all $n > n(\eps,k)$.
Now, using the variational definition of the least eigenvalue of a symmetric matrix, we have
\begin{equation} \label{eq:lminlmin} \l_{\min}(\Sigma_0) \geq \l_{\min}(\Sigma') - \max_{v , |v|_2 =1} v' (\Sigma_0-\Sigma') v  = \l_{\min}(\Sigma') + O_{k}(\delta),\end{equation}
where we have used Lemma~\ref{lem:main-compare} to bound 
\[\max_{v , |v|_2 =1} v' (\Sigma_0-\Sigma') v = \|\Sigma_0-\Sigma'\|_{op} = O_{k}(\delta). \] 
Using \eqref{eq:lmin-lb} and \eqref{eq:lminlmin} and choosing $\delta>0$ to be sufficiently small compared to $k$ and $\eps$, gives
\begin{equation}\label{eq:detSigma0} \det(\Sigma_0) \geq (\l_{\min}(\Sigma'))^{2k}  = \Omega_{k,\eps}(1), \end{equation}
for sufficiently small $\delta$ and $n > n(k,\eps)$.

Putting the pieces together, we use \eqref{eq:det-split} along with \eqref{eq:det-rescale} to express
\[ \det\Sigma  = n^T (\det \Sigma_0)(\det \Delta )^{-2}.\]
We then apply Lemma~\ref{lem:det-divided-diff} and \eqref{eq:detSigma0} to write, for $c_{k,\eps}>0$, 
\begin{equation}\label{eq:detSigma'} \det\Sigma \geq  c_{k,\eps} n^T \prod_{i=1}^r \left\lbrace \prod_{0 \leq a < b \leq k(i)}(x_{i,b} - x_{i,a})^2 
\prod_{0 \leq a < b \leq k'(i)}(y_{i,b} - y_{i,a})^2 \right\rbrace . \end{equation}
So, finally, using the clear fact $1 \geq n^2 \min \{|x-x'|^2,n^{-2}\}$, we rewrite \eqref{eq:detSigma'} as
\[  \det\Sigma \geq  c_{k,\eps} n^{2k^2} \prod_{i<j}\min\{ |x_i-x_j|, n^{-1} \}^2 \cdot \min\{ |y_i-y_j|, n^{-1} \}^2,  \]
for all $n > n(\eps,k)$, thus completing the proof of Lemma~\ref{lem:det-close}.
\end{proof}
\vspace{4mm}

\noindent 


To prove Lemma~\ref{lem:det-LB-general}, all that remains is to apply Lemma~\ref{lem:det-close} at an appropriate ``scale''. 
That is, we need an appropriate distinction between ``near'' and ``far'' points.  The following will provide us with just this.

\begin{lemma}\label{lem:scale} Let $\{\delta_k \}_k \subseteq \R^+$ be a decreasing sequence and let $S \subseteq \R$ have $|S| = s$. Then there exists $R \subseteq S$ with $|R| = r$ so that every $z \in S$ has $d(z,R) \leq \delta_{r}$ and any distinct $x,y \in R$ have $|x-y| > \delta_{r-1}$.
\end{lemma}
\begin{proof}
	Let $r$ be the largest integer for which there exists a set $R := \{x_1,\ldots,x_r \} \subseteq S$ so that  $|x_i - x_j| \geq \delta_{r-1}$, for all $i\not= j$. For a contradiction, assume there exists a point $z \in S\setminus R$ for which $|z - x_i| \geq \delta_r$ for all $i \in [r]$. But then if we take $x_{r+1} = z$
	the set $\{x_1,\ldots,x_{r+1}\}$ contradicts the maximality of $r$, as $|x_i-x_j| \geq \delta_r$, for all $i\not= j$. Thus we conclude that all points of $S\setminus R$ are within $\delta_{r}$ of a point in $R$.\end{proof}

\vspace{4mm}

We now prove Lemma~\ref{lem:det-LB-general}, the main result of this section. 

\begin{proof}[Proof of Lemma~\ref{lem:det-LB-general}]
We define a decreasing sequence $\{ \delta_m \}_m \subseteq \R_{>0}$ inductively by first setting $\delta_1 = \delta(k,1)$, where $\delta(\cdot,\cdot)$ is the constant that appears in Lemma~\ref{lem:det-close}. Then, assuming that 
$\delta_{t}$ has been chosen, choose $\delta_{t+1} = \delta(k,\delta_t)/2$.
We now apply Lemma~\ref{lem:scale} to find a value of $t$ and a set of points $z_1,\ldots,z_t$ so that $d(z_i,z_j) > \delta_t/n$ for $i \not= j$ 
and all elements of $\{x_1,\ldots,x_{\ell},y_1,\ldots y_k \}$ have distance $\leq \delta_{t+1}/n$ to one of the values $\{z_i\}$.
We may now apply Lemma~\ref{lem:det-close} with $\eps = \delta_{t}$ and $\delta = \delta_{t+1}$ to obtain
\[ \det \Sigma_{k}(x,y)\geq c_{k}n^{2k^2} \prod_{i < j} \left(\min\{x_j - x_i,1/n\}^2 \cdot \min\{y_j - y_i,1/n\}^2\right), \]
for all large enough $n > n(k)$.
\end{proof}

\section{Proof of Lemma~\ref{th:density-bound}}

We now have everything we need to prove Lemma~\ref{th:density-bound} which, as we have seen in Section~\ref{sec:using-KR}, implies Theorem~\ref{thm:main1}, our main theorem.

\begin{proof}[Proof of Lemma~\ref{th:density-bound}]
Recalling the definition of $p_k(\bx,\by)$ from \eqref{eq:def-pk}, and applying Lemma~\ref{cor:num-UB} and Lemma~\ref{lem:det-LB-general} to the numerator and denominator, respectively, we have 
\begin{align}\label{eq:finalpkineq} p_k(\bx,\by) &= \frac{\alpha_{k}(\bx,\by)}{(2\pi)^{k} |\Sigma|^{1/2}} \leq  C_{k} n^{k^2 + k} \left(\prod_{i < j} \min\{|x_i - x_j|,n^{-1}  \}\cdot \min\{|y_i - y_j|,n^{-1}  \} \right)  \\
&\leq C_k n^{k^2 + k} \cdot n^{-k(k-1)} = C_k n^{2k} \nonumber \end{align} 
 as claimed. \end{proof}

\vspace{4mm}

Notice that in the proof of Lemma~\ref{th:density-bound}, we actually obtain a sharper bound on $p_k$, at \eqref{eq:finalpkineq}.
It is our suspicion that this bound is actually the correct one. 

\begin{conj}
	For all $\bx, \by \in \T_0$, \[  p_k(\bx,\by)  = \Theta_k\left( n^{k^2 + k} \left(\prod_{i < j} \min\{|x_i - x_j|,n^{-1}  \}\cdot \min\{|y_i - y_j|,n^{-1}  \} \right) \right)\,.\]
\end{conj}

A resolution to this conjecture would then determine probabilities of events like those considered in Lemma~\ref{lem:density-bound-restatement},
up to constant factors.

\appendix

\section{Details from Section~\ref{sec:unit-circle-redux}} \label{sec:details}

In this appendix we will tie up a few loose ends from Section~\ref{sec:unit-circle-redux} by proving
Lemmas~\ref{lem:good-complex}, \ref{lem:circle-good}, \ref{lem:endpoints}, \ref{lem:close-roots} and \ref{lem:close-roots-circle}.
Perhaps unsurprisingly, we employ yet another Kac-Rice formula, which is actually much simpler than those which we have been working with.
We won't need to dive too deeply into the behavior of these integrals here; it will just allow us to get upper bounds on the number of roots in a region $U$ when we have an upper bound on $f'$ in $U$.

\begin{lemma}[Complex Kac-Rice] \label{lem:complex-KR}
	Let $V \subset \C$ be an open set and let $U \subset \C$ be a compact set.  Then 
	$$\E\left[| \{z \in U : f(z) = 0, f'(z) \in V \}|  \right] = \int_U \frac{\E[ |f'(z)|^2 \cdot \one\big[ f'(z) \in V \big]\,|\, f(z) = 0  ]  }{2\pi|\det \Cov(f(z))|^{1/2} }\,dz , $$
	where the integral is with respect to $2$-dimensional Lebesgue measure.
\end{lemma}

\noindent Lemma \ref{lem:complex-KR} is derived from a more general Kac-Rice formula in Appendix \ref{sec:KR-details}.

\subsection{Roots close to the circle} 
For the logical flow of this appendix, it actually makes sense to start by proving Lemma~\ref{lem:endpoints}, which will allow us to ignore
the special behavior that $f$ has near the real axis. For this, we study the covariance of $f$ close to the real axis.
\begin{lemma}\label{lem:det-close-to-axis}
	For $M \geq 0$, $x \in [0,1]$ and $|\rho|\leq M/n^2$ we have $$\det \Sigma := \det \Cov( f((1+\rho)e^{ix})) =\Omega\left(\min\{ n^4 x^2, n^2 \}    \right)\,.$$
\end{lemma}
\begin{proof}
	For $\gamma > 0$ to be determined later, we consider three regions for $x$: 
\[ x \leq \gamma/n, \qquad x  \in [\gamma/n, n^{-2/3}], \qquad x > n^{-2/3}. \]
Starting with the region $x \leq \g/n$, fix $\delta > 0$ and note that 
$$\Sigma_{1,1} = \sum_{k = 0}^n (1+\rho)^{2k}\cos^2(kx) \geq (1 - |\rho|)^n n\left(\int_0^1 \cos^2( nx t)\,dt \right) \geq n( 1 - \delta), $$ where this last inequality holds provided $n$ is sufficiently large and $\g$ is small compared to $\delta$.  Similarly, bound $$\Sigma_{2,2} = \sum_{k = 0}^n (1 + \rho)^{2k}\sin^2(kx) \geq (1 - \delta)\frac{n^3 x^2}{6} \text{ and }\Sigma_{1,2} = \sum_{k = 0}^n (1 + \rho)^{2k}\sin(kx)\cos(kx) \leq (1 + \delta)\frac{n^2 x}{4}\,.$$
	
	Putting these calculations together tells us that for $x \leq \gamma/n$ we have $$\det(\Sigma) \geq \frac{n^4 x^2}{12}(1 + o_{\delta \to 0}(1))\,.$$
	Thus for $\gamma$ sufficiently small and $x < \gamma/n$ the Lemma is proven.
	
	Turning to the region $x \leq n^{-2/3}$, we use the trapezoidal rule to see 
	$$\Sigma_{1,1} = \sum_{k = 0}^n(1 + \rho)^{2k} \cos^2(kx) = n\left( \int_0^1 \cos^2(nxt)\,dt\right) + O(1)\,.$$
	Evaluating the integral, setting $y = nx$ and doing the same for the other two entries of $\Sigma$ shows
	 $$\Sigma = n\begin{bmatrix}
	\frac{\cos(y)\sin(y) + y}{2y} & \frac{\sin^2(y)}{2y}\\
	\frac{\sin^2(y)}{2y} &  \frac{y - \cos(y)\sin(y) }{2y}
	\end{bmatrix} + O(1)\,.$$
	Thus for $nx = y \in [\gamma,n^{1/3}]$, $$\det(\Sigma) \sim \frac{n^2}{4}\left( \frac{\cos(y)^2 + y^2 - 1}{y^2}\right) = \Omega(n^2)\,.$$
	
	For the remaining region $x > n^{-2/3}$, bound $\Sigma_{1,1} = \Omega(n), \Sigma_{2,2} = \Omega(n)$ and $$\Sigma_{1,2}= \frac{1}{2}\left(\frac{\cos(x)(1 - \cos(x(n+1))^2) }{\sin(x)} - \sin(x(n+1))\cos(x(n+1))\right) = O(x^{-1}) = O(n^{2/3})\,.$$
	
	This shows $\det(\Sigma) = \Omega(n^2)$ in this regime.
\end{proof}

\begin{proof}[Proof of Lemma \ref{lem:endpoints}]
First note that $f$ has real coefficients and that the law of $f(-z)$ is equal to that of $f(z)$. Therefore it is sufficient to show that the set 
\[ S:= \left\lbrace z \in \cA_n\left( [-M,M]\right) : 0 \leq \arg(z) \leq n^{-\eps} \right\rbrace, \]
is zero free, with high probability.

Set $z_k =\exp(i x_k)$, where $x_k = k/n^2$ for $k \in \{0,1,\ldots, n^{2-\eps} \}$.  We will first show that $|f|$ is at least $n^{-1/2}\log n$ at each of each of the points $z_k$. We will then show that this implies that $f$ is non-zero in the balls $B_{2M/n^2}(z_k)$, which collectively cover $S$.  We attack this former point first.

For this we show
\begin{equation} \label{eq:lem-end-points-goal1} \sum_{k = 0}^{n^{2-\eps}} \P(|f(z_k)| \leq n^{-1/2}\log n) = o(1).\end{equation}
To see this, note 
\[  \P(|f(z_k)| \leq n^{-1/2}\log n) \leq \P(|\re(f(z_k))|, |\im(f(z_k))| \leq n^{-1/2}\log n )  = O\left( n^{-1} (\log n)^2 \det(\Sigma)^{-1/2} \right)\]
and then apply Lemma~\ref{lem:det-close-to-axis} to obtain
\begin{equation}\label{eq:P(f(z_k))small} \P(|f(z_k)| \leq n^{-1/2}\log n) = O\left(  n^{-2}(\log n)^2\max\{(nx_k)^{-1},1\}  \right)= O\left(  n^{-2}(\log n)^2\max\{n/k,1\}  \right).  \end{equation}
We also have, by direct computation,
\[ \P(|f(z_0)| \leq n^{-1/2}\log n ) = O\left( n^{-1}\log n \right)\,. \] So summing this along with \eqref{eq:P(f(z_k))small} over $k \in [0,n^{2-\eps}]$, yields \eqref{eq:lem-end-points-goal1}.
	
We now turn to show that if $|f(z_k)| \geq n^{-1/2}\log n$ then $f$ is zero-free in $B_{2M/n^2}(z_k)$. For this, note that,
by Fact \ref{facts}, we may condition on the event $|f'(z)| \leq (4M)^{-1}n^{3/2}\log n $ for all $|z| \leq 1 + 2M/n^2$ at only an additive loss in $o(1)$ in our probability calculations.

For any $z \in B_{2M/n^2}(z_k)$ the mean-value theorem for complex functions implies that there are points $\xi_1$ and $\xi_2$ on the line segment connecting $z$ and $z_k$ so that
\[ f'(\xi_1) = \re\left(\frac{f(z) - f(z_k)}{z-z_k} \right)  \qquad f'(\xi_2) = \im\left(\frac{f(z) - f(z_k)}{z-z_k} \right)  \,. \]
This implies that 
$$\left|\frac{f(z) - f(z_k)}{z - z_k}\right| \leq \sqrt{2}(4M)^{-1} n^{3/2} \log n$$ and so 
$$|f(z) - f(z_k)| \leq (2n)^{-1/2}\log n  \,.$$
Thus, if $|f(z_k)| \geq n^{-1/2}\log n$, $f(z)$ is non-zero. This, along with \eqref{eq:P(f(z_k))small}, implies that $f$ has no zeros in
 $\bigcup_{k} B_{2M/n^2}(z_k) \supseteq S$ with high probability.
\end{proof}

\subsection{Lemmas~\ref{lem:good-complex},~\ref{lem:circle-good} and \ref{lem:close-roots}}

\begin{proof}[Proof of Lemma~\ref{lem:good-complex} ] Call a zero $\z \in Z(f)$ \emph{bad} if $||\z|-1|\leq n^{-2}(\log n)^{1/4}$ and
\[ |X'(\arg(\z))|,|Y'(\arg(\z))| \leq  n^{3/2}/\log n.\] 
	Note that by Lemma \ref{lem:endpoints}, we need only to show there are no bad zeros $\z$ with $\arg(\z) \in \T_0$, with high probability.  
	We start by counting the number of zeros with a slightly different property and then relate these to bad zeros.  
	
	Say that a zero $\z \in Z(f)$ is \emph{nearly bad} if $||\z|-1|\leq n^{-2}(\log n)^{1/4} $ and
	\[ |\re(\z f'(\z))| \leq  2n^{3/2}/\log n \,  \textit{  or  } \, |\im(\z f'(\z))| \leq 2n^{3/2}/\log n \] and let $b$ be the number of nearly bad zeros. Put
	\[ U := \{x + iy \in \C: |x| < 2n^{3/2}/\log n \text{ or }|y| < 2n^{3/2}/\log n  \}, \] set 
	\[ V = \{z : ||z|-1|\leq n^{-2}(\log n)^{1/4}, \, |\arg(z)| \geq n^{-1/2} \text{  and  } |\arg(-z)| \geq n^{-1/2}\} \] and 
	apply Lemma~\ref{lem:complex-KR} to the function $g(z) = zf(z)$ to get an integral form for the number of $\z \in V$ with $g(\z) =0$ and $g'(\z)  \in U$.
	Since $\z \not = 0 $ and $g'(z) = f(z) + zf'(z)$ this is the same as counting $\z \in V$ with $f(\z) = 0$ and $\z f'(\z)  \in U$. That is,
	\begin{equation}\label{eq:bI} b = \E\left[| \{z \in V : f(z) = 0, zf'(z) \in U \}|  \right] 
	= \int_V \frac{\E[ |g'(z)|^2 \cdot \one\left[ g'(z) \in U\right] \,|\, g(z) = 0 ]  }{2\pi\cdot \det \Cov(g(z))^{1/2} }\,dz\, . \end{equation}
 Since we are working away from the real axis, Lemma \ref{lem:covar-approx} and Fact \ref{fact:approx-int} apply and tell us that the denominator in this integrand satisfies 
	$$ \left|\det \Cov(g(z))\right|^{1/2} = |z|\left|\det \Cov(f(z))\right|^{1/2} = \Theta(n)\,.$$  
	To bound the numerator, we apply Cauchy-Schwarz \begin{align*}
	\E[ |g'(z)|^2 \cdot \one\left[g'(z) \in U\right]\,|\, g(z) = 0  ]
	 &\leq \left(\E[ |g'(z)|^4 \,|\, g(z) = 0  ]\right)^{1/2}\left(\P(g'(z) \in U\,|\,g(z) = 0)\right)^{1/2} \\
	&=O\big(n^3(\log n)^{-1/2}\big)\,.\end{align*}
Where we have used 
\[ \E[\, |g'(z)|^4 \,|\, g(z) = 0  ] \leq C'\big(\E\, |g'(z)|^2 |\, g(z) = 0 \big)^2 \leq \big(\E\, |g'(z)|^2 \big)^2 = O(n^{6}), \]
where the first and second inequalities follow from properties of Gaussian Random variables (Fact~\ref{fact:gaussian}). We have also used
\[ \left(\P(g'(z) \in U\,|\,g(z) = 0)\right)^{1/2} \leq C|U|\left(\Var( |g'(z)| \,|\,g(z) = 0)\right)^{-1/2} <(\log n)^{-1},\]
where the first inequality is again due to a property of Gaussian random variables, the second inequality holds due to \eqref{eq:early-approx2} in Lemma~\ref{lem:early-approx}. 

Thus, using \eqref{eq:bI}, along with the fact that $|V| = O(n^{-2}(\log n)^{1/4})$, we see
\[ b =   O\big( n^3 (\log n)^{-1/2} \cdot n^{-1} \big) \int_V |dz| = o(1)\, \]
and so the probability there is a nearly bad zero $\z$ is $o(1)$.

We now show that the above implies that there are no \emph{bad} zeros $\z$ with high probability. So suppose that $\z = e^{i\t}\rho \in \cA_f(I)$ is bad and in particular: $f(\z) = 0 $ and $|X'(\t)| < n^{3/2}/\log n$ (the case with $X$ replaced with $Y$ is similar). This implies, by mean value theorem, that there is a $\rho_0 \in [0,n^{-2}\max{I}]$ so that
	\[ |X'(\t,\rho)| \leq |X'(\t,0)| + \rho_0 |X''(\t,0)| \leq (1+o(1)) n^{3/2}/\log n ,\] with high probability. 
	Since $X'(\t,\rho) = \re(\z f'(\z) )$ we see that 
	$\z$ is a zero of $f$ that satisfies $|\re(\z f'(\z))| \leq 2n^{3/2}/\log n $ and therefore is a nearly bad zero. Since the probability there exists a nearly bad zero is $o(1)$, the probability there is a bad zero is $o(1)$.
\end{proof}

\begin{proof}[Proof of Lemma \ref{lem:circle-good}] Say that $(x_0,y_0) \in \cC_f(I)$ is \emph{bad} if $\min\{|X'(x)|,|Y'(y)|\} < 2 n^{3/2}(\log n)^{-1}$.
	 We calculate the expected number of bad pairs
\[ b := \E \left[\left| \left\lbrace (x,y) \in \mathcal{C}_f(I)  :\, \min\{|X'(x)|,|Y'(y)|\} < 2 n^{3/2}(\log n)^{-1}   \right\rbrace \right|\right]. \] Now use that fact that $|x-y|\leq n^{-2}(\log n)^4$ and $d(\{x,y\},\pi\Z) \geq n^{-1/2}$ for all $(x,y) \in \cC_f(I)$ and argue as in Lemma \ref{lem:eval-mean} to bound

$$b \leq Cn^2 \int_{x \in \T_0} \int_{|r| \leq \eta} \E\left[ |Z_1|\cdot |Z_2| \one\left\{ \frac{n^2 r Z_1 Z_2}{Z_1^2 + Z_2^2} \in I\right\} \one\{\min\left\{ |Z_1|,|Z_2|\} \leq 2 (\log n)^{-1} \right\} \right]\,dr\,dx$$
where $\eta = (\log n)^4$ and $(Z_1,Z_2)$ is a standard 2-dimensional Gaussian.  Anticipating an application of Fubini's theorem, bound $$\int_{|r| \leq \eta} \one\left\{ \frac{n^2 r Z_1 Z_2}{Z_1^2 + Z_2^2} \in I\right\}\,dr \leq \int_{-\infty}^\infty \one\left\{ \frac{n^2 r Z_1 Z_2}{Z_1^2 + Z_2^2} \in I\right\}\,dr = \frac{|I|}{n^2 |Z_1|\cdot |Z_2|}(Z_1^2 + Z_2^2)\,.$$
Combining with the bound on $b$ then gives \begin{align*}b &\leq C n^2 \int_{x \in \T_0} \E\left[|Z_1| \cdot |Z_2| \frac{|I|}{n^2 |Z_1| \cdot |Z_2|}(Z_1^2 + Z_2^2) \one\{\min\{|Z_1|,|Z_2|\} \leq 2 (\log n)^{-1}  \} \right]\,dx \\
&\leq  C \pi |I| \E[(Z_1^2 + Z_2^2) \one\{\min\{|Z_1|,|Z_2|\} \leq 2 (\log n)^{-1}  \}  ] \\
&= o(1)\,.
\end{align*}

And so, for each pair $(x_0,y_0) \in \mathcal{C}_f(I)$, we have $|X'(x_0)|,|Y'(y_0)| > 2 n^{3/2}/\log n$, with high probability. Arguing as in the end of the proof of Lemma~\ref{lem:good-complex} with the mean value theorem shows $|X'(y_0)|,|Y'(x_0)| > n^{3/2}/\log n$ with high probability.
\end{proof}

\vspace{4mm}

\begin{proof}[Proof of Lemma \ref{lem:close-roots}] 
We first apply Lemma \ref{lem:endpoints} to see that there are no roots $\z$ with $\arg(\z) \in \T \setminus \T_0$, with high probability. Thus we may restrict ourselves to the region $\arg(z) \in \T_0$.
Say that $\z_1,\z_2$  is a \emph{bad pair} if 
$$\z_1,\z_2 \in \{ z \in \C : ||z|-1| \leq n^{-1} \text{ and }  \arg(z) \in \T_0 \} =: V$$ are distinct roots of $f$ with $|\zeta_1 - \zeta_2| \leq n^{-2}(\log n)^{10}$. 

To show that there are no bad pairs, we work with a slightly different notion: say that a $\z$ is \emph{rotten} if $\z \in V$ is a zero of $f$ and $|f'(z)| \leq n^{1/2}(\log n)^{13}$. Let us first observe that there is no rotten root with high probability. Indeed, applying Lemma~\ref{lem:complex-KR} and noting $\det \Cov f(z) = \Omega(n^2)$ for $\arg(z) \in \T_0$, we express the expected number of rotten points as 
\[ \int_{V} \frac{\E[|f'(z)|^2 \cdot \one\left(|f'(z)| \leq n^{1/2}(\log n)^{13} \right) \,|\,f(z) = 0 ]}{2\pi (\det \Cov(f(z)))^{1/2} }\,dz = O\left( n (\log n)^{26} \cdot n^{-1} \cdot |V| \right), \]
which is $o(1)$ and thus there are no rotten points with high probability. 

We now see that a bad pair implies a rotten point;
suppose $\z_1,\z_2$ is a bad pair. By the mean-value theorem for complex functions, there exists points $u$ and $v$ on the line segment connecting $\zeta_1$ and $\zeta_2$ so that \begin{align*}
	\re f'(u) =\re\left( \frac{f(\zeta_1) - f(\zeta_2)}{\zeta_1-\zeta_2} \right) &=  0 \\
	\im f'(v) =\im\left( \frac{f(\zeta_1) - f(\zeta_2)}{\zeta_1-\zeta_2} \right) &= 0 \,.
	\end{align*}	
Recall that $\max_{|z| \leq 1+n^{-1}}|f^{(2)}(z)| = O(n^{5/2}\log n)$ with high probability and so we may apply mean value theorem again to see that 
for all $z$ with $|z- \z_1| < n^{-2}(\log n)^{11}$ we have 
	$$|f'(z) - f'(\zeta_1)|= O(n^{-2}(\log n)^{11}\cdot n^{5/2} \log n )  = O( n^{1/2}(\log n)^{12}),$$
with high probability. This shows that $|f'(\zeta_1)| = O(n^{1/2} \log^{12} n )$ and thus $\z_1$ is rotten. This shows that if $\z_1,\z_2$ is a bad pair 
then one of $\{\z_1,\z_2 \}$ is a rotten point, up to a set of measure $o(1)$. Thus there are no bad pairs with probability $o(1)$.
\end{proof}

\vspace{3mm}

\begin{proof}[Proof of Lemma \ref{lem:close-roots-circle}]
	Suppose there exist distinct $x_1,x_2 \in \T_0$ with $|x_1 -x_2| \leq n^2(\log n)^{10}$ and $g(x_1) = g(x_2) = 0 $. Then there exists $\xi$ with $g'(\xi) = 0$ and $|\xi - x_1| \leq n^{-2} (\log n)^{10}$ by the mean-value theorem.  Let $E$ be the event that 
 $\max_{x \in [0,2\pi]} |g''(x)| \leq Cn^{5/2}(\log n)^{1/2}$ and note that 	
	by the Salem-Zygmund and Bernstein inequalities, $\PP(E) = 1-o(1)$. Now, conditioned on $E$, the mean-value theorem implies that $|g'(x_1)| \leq n^{1/2}(\log n)^{11}$. By the Kac-Rice formula (Lemma \ref{lem:kac-rice}) 
\[ \E\, \left|\left\{ x \in \T_0 : g(x) = 0, |g'(x)| \leq n^{1/2}(\log n)^{11}   \right\} \right| 
= \int_{\T_0} \frac{\E[|g'(x)| \one\left( |g'(x)| \leq n^{1/2}(\log n)^{11}\,|\,g(x) = 0   \right)]}{(2\pi \Var(g(x)) )^{1/2}}\,dx. \]
We then bound the right hand side by
\[ n^{1/2} (\log n)^{11} \int_{\T_0} \frac{\P[|g'(x)| \leq n^{1/2}(\log n)^{11}\,|\,g(x) = 0 ]  }{(2\pi \Var(g(x)) )^{1/2}}\,dx \\
	= O\left( n^{1/2} (\log n)^{11} \cdot (\log n)^{11} n^{-1} \cdot n^{-1/2} \right), \,\]
	which tends to $0$ as $n\rightarrow \infty$, as desired.

\end{proof}

\section{Proof of Lemma~\ref{lem:our-KR-form}, our Kac-Rice-type formula}\label{sec:KR-details}

At first, formulas such as \eqref{eq:our-KR} can appear a bit unwieldy and their utility opaque. 
So we take a moment, before diving into the proof of Lemma~\ref{lem:our-KR-form}, to say a few words about where these integrals come from.
For this, we take a simplified version of the equation in Lemma~\ref{lem:our-KR-form}; \begin{equation}\label{eq:basic-KR}
\E[|\{x \in [0,2\pi] : X(x) = 0  \}  ] = \int_{[0,2\pi]} \frac{\E{[|X'(x)| \,|\,X(x) = 0] }}{(2\pi)^{1/2} \Var(X(x))^{1/2}}\,dx.
\end{equation}
Let $p(x)$ denote the integrand on the right hand side of \eqref{eq:basic-KR} and let us fix $x \in [0,2\pi]$ and take $\eps > 0$
to be extremely small. We will now observe that  the probability that there is a zero in the interval $(x-\eps,x+\eps)$ is roughly $2\eps \cdot p(x)$. Of course, if we can show this to be true, linearity of expectation immediately gives us \eqref{eq:basic-KR}. 

Now, when $\eps$ is sufficiently small, the function $X(\cdot)$ is approximately linear on the interval $(x-\eps,x+\eps)$ with slope $ \approx X'(x)$. It is therefore easy to see when there is a zero in this interval: there is a zero in $(x-\eps,x+\eps)$ roughly when $X(x-\eps) \in (-2\eps X'(x),0)$, (assuming $X'(x) >0$ without loss of generality). Using the properties of Gaussian random variables, one can calculate 
\[ \PP\big( X(x-\eps) \in (-2\eps X'(x),0) \big\vert\, X'(x) \big) \approx \frac{2\eps |X'(x)|}{ \sqrt{2\pi \Var(X(x))}}. \]
We now want to average over all possible values of $X'(x)$ to ``eliminate'' the conditional expectation. Here we note that $\eps$ can be taken small (compared to everything) and therefore we can restrict to considering the case when $X(x)$ is small. Thus we condition  on $|X(x)| \leq \eta$, for some small $\eta$ and take expectations on both sides. Finally, dividing this result by by $2\eps$ gives exactly the Kac-Rice density, once we send $\eta$ and $\eps$ to zero.

Of course, this is only a sketch and turning this into a genuine proof requires a bit of work. Instead of doing this here
we derive our results from a much more general result on Gaussian processes.

\subsection{Deriving Lemmas \ref{lem:our-KR-form} and \ref{lem:complex-KR}}
Rather than prove our Kac-Rice formulas from scratch, we derive them from a general multivariate formula of Aza\"is-Wschebor. 
For a smooth function $F:\R^n \to \R^n$, we write $F'$ for the Jacobian of $F$, given as an $n \times n$ matrix.

\begin{theorem}[Theorem $6.4$ \cite{azais2009level}]\label{th:general-KR}

	Let $U \subset \R^d$ be an open set and let $Z:U \to \R^d$ be Gaussian. Let $Z(t)$ be such that  \begin{enumerate}
		\item the function $t \mapsto Z(t)$ is almost-surely of class $C^1$;
		\item For each $t \in U$, the matrix $\Cov(Z(t))$ is positive definite;
		\item $\P(\exists~t \in U, Z(t) = 0, \det(Z'(t)) = 0) = 0$.
	\end{enumerate}
	Assume further that one has another random function $W:U \to \R^n$ satisfying: \begin{enumerate}
		\item The function $t \mapsto W(t)$ is continuous almost-surely.
		\item For each fixed $t \in U$ the random process $ (Z(s),W(t))_{s \in U}$ is Gaussian.
	\end{enumerate}
	Then for every continuous bounded function $g: U \times \R^n \to \R$ and for every compact $I \subset U$ we have \begin{align*}
	\E\left(\sum_{t \in I, Z(t) = 0} g(t,W(t))\right) = \int_I \frac{\E\left(| \det(Z'(t))|\, g(t,W(t))\, \Big\vert \, Z(t) = 0\right)  }{(2\pi)^{n/2} \det \Cov(Z(t))^{1/2} }\,dt\,.
	\end{align*}
\end{theorem}

\vspace{2mm}

\begin{proof}[Proof of Lemma \ref{lem:complex-KR}]
	Identify $\C$ with $\R^2$ and so we may view $f$ as a function from $\R^2$ to $\R^2$.  Seeking to apply Theorem \ref{th:general-KR}, set $W(t) = f'(t)$ where $f'$ is the complex derivative of $f$ viewed as a function on $\R^2$.  Let $g_\alpha$ be a sequence of 
	continuous functions that increases monotonically to the indicator of $V$: $g_{\alpha} \uparrow \one_V$. For all $\alpha$ we have  

$$\E\left(\sum_{t \in U, f(t) = 0} g_\alpha(f'(t))\right) = \int_U \frac{\E\left( | \det(J(t))|\, g_\alpha(f'(t))\,\Big\vert \, f(t) = 0 \right)  }{2\pi \det \Cov(f(t))^{1/2} }\,dt$$
where we write $J(t)$ to be $2 \times 2$ Jacobian of $f$, where we view $f$ as a function in $\R^2$.  Writing $f = u + iv$, the Cauchy-Riemann equations imply  $$\det(J(t)) = u_x(t)v_y(t) - u_y(t)v_x(t) = (u_x(t))^2 + (v_x(t))^2 = |f'(t)|^2\,.$$
	Taking $\alpha \to 0$ sends $g_\alpha \uparrow \one_V$ and so applying dominated convergence theorem completes the proof.
\end{proof}

\vspace{4mm}

\noindent Lemma \ref{lem:our-KR-form} will follow from the Kac-Rice form below.
Recall that for $k \in \N$ and $V \subset \R^{2k}$ we have  \begin{equation}\label{eq:KR}
	p_k(\mathbf{x},\mathbf{y},V) := 
\frac{\E\left[ \left|X'(x_1)\cdots X'(x_k) Y'(y_1)\cdots Y'(y_k) \right| \varphi_V\,|\,X(x_i)=0, Y(y_j) = 0 \text{ for all }i,j \right]}{(2\pi)^{k} |\Sigma|^{1/2}},
	\end{equation}
	where $\Sigma$ is the $2k\times2k$ covariance matrix $\Cov (X(x_i),Y(y_i))$ and $\varphi_V$ is the indicator for the event $(X'(x_1),\ldots,X'(x_k),Y'(y_1),\ldots,Y'(y_k)) \in V$.  For any compact set $T \subset \T^{2k}$, let $N_T$ be the number of pairwise distinct tuples of pairs $((x_1,y_1),\ldots,(x_k,y_k)) \in T$ where $X(x_i) = 0, Y(y_j) = 0$ for all $i,j$ and $\varphi_V = 1$. 

\begin{lemma}\label{lem:kac-rice} We have 
	 $$ \E[N_T] = \int_T p(\bx,\by,V)\,d\bx\,d\by\,.$$
\end{lemma}
\begin{proof}
	We adapt an argument from \cite[Theorem 11.5.1]{adler2009random}. Fix $\delta > 0$ and let 
	\[ Z(t_1,\ldots,t_{2k}) = (X(t_1),\ldots,X(t_k),Y(t_{k+1}),\ldots, Y(t_{2k})).
	\]  Note that $Z'$, the Jacobian of $Z$, is a diagonal $2k \times 2k$ matrix with diagonal 
	\[ (X'(t_1), \ldots, X'(t_k), Y'(t_{k+1}),\ldots, Y'(t_{2k})).\] For a set $T \subset [0,2\pi]^{2k}$, define\footnote{ If $X$ is a set, we let $X^{(k)}$ to be the set of $k$ element subsets of $X$.}
	\[ T_\delta := \{t \in T : |t_i - t_j| \geq \delta \text{ for all } \{i,j\} \in [k]^{(2)} \cup [k+1,2k]^{(2)} \}.\] 
	 As in the proof of Lemma \ref{lem:complex-KR}, let $\{g_\alpha \}_{\alpha}$ be a sequence of continuous functions so that $g_\alpha \uparrow \one_V$ as $\alpha \to 0$. By Theorem \ref{th:general-KR} we have 
	\begin{equation} \label{eq:dervKR} \E\left(\sum_{t \in T_\delta, Z(t) = 0} g_\alpha(Z'(t))\right) = \int_{T_\delta} \frac{\E\left( | \det(Z'(t))|\, g_\alpha(Z'(t))\,\Big\vert \, Z(t) = 0 \right)  }{(2\pi)^{k} \det \Cov(Z(t))^{1/2} }\,dt\,.\end{equation}
	Taking $\alpha \to 0$ implies $g_\alpha \uparrow \one_V$. Monotone convergence theorem allows us to swap $g_{\alpha}$ for $\one_V$ on both sides of \eqref{eq:dervKR}. Now, sending $\delta \to 0$, and again using monotone convergence theorem, allows us to replace $T_{\delta}$ with $T$ on both sides of \eqref{eq:dervKR}. Noting that $\det(Z'(t)) = \prod_{i=1}^k X(t_i)Y(t_{i+k})$, completes the proof. 
\end{proof}

\begin{proof}[Proof of Lemma~\ref{lem:our-KR-form}]
Take $T = \T_0^{2k}$ and $V = \cC_f(U)$ in Lemma~\ref{lem:kac-rice}. Set $R := \mu_f(U) = |\cC_{f}(U)|$ and note that $(R)_k$ counts
the number of pairwise distinct tuples of pairs $((x_1,y_1),\ldots,(x_k,y_k)) \in \cC_{f}(U)$. So
\[ \E\, (R)_k = N_T = \int_{\T_0} p_k(\bx,\by,U)\, d\bx\, d\by, \]
as desired.

\end{proof}

\bibliographystyle{abbrv}
\bibliography{Bib}

\begin{thebibliography}{10}

\bibitem{adler2009random}
R.~J. Adler and J.~E. Taylor.
\newblock {\em Random fields and geometry}.
\newblock Springer Science \& Business Media, 2009.

\bibitem{arnold}
L.~Arnold.
\newblock \"{U}ber die {N}ullstellenverteilung zuf\"{a}lliger {P}olynome.
\newblock {\em Math. Z.}, 92:12--18, 1966.

\bibitem{azais-universility}
J.-M. Aza{\"\i}s, F.~Dalmao, J.~Le{\'o}n, I.~Nourdin, and G.~Poly.
\newblock Local universality of the number of zeros of random trigonometric
  polynomials with continuous coefficients.
\newblock {\em arXiv preprint arXiv:1512.05583}, 2015.

\bibitem{azais2009level}
J.-M. Aza{\"\i}s and M.~Wschebor.
\newblock {\em Level sets and extrema of random processes and fields}.
\newblock John Wiley \& Sons, 2009.

\bibitem{billingsley}
P.~Billingsley.
\newblock {\em Probability and measure}.
\newblock Wiley Series in Probability and Mathematical Statistics: Probability
  and Mathematical Statistics. John Wiley \& Sons, Inc., New York, second
  edition, 1986.

\bibitem{bloch-polya}
A.~Bloch and G.~P\'{o}lya.
\newblock On the {R}oots of {C}ertain {A}lgebraic {E}quations.
\newblock {\em Proc. London Math. Soc. (2)}, 33(2):102--114, 1931.

\bibitem{deBoor}
C.~de~Boor.
\newblock Divided differences.
\newblock {\em Surv. Approx. Theory}, 1:46--69, 2005.

\bibitem{dunnage}
J.~Dunnage.
\newblock The number of real zeros of a random trigonometric polynomial.
\newblock {\em Proceedings of the London Mathematical Society}, 3(1):53--84,
  1966.

\bibitem{erdosOfford}
P.~Erd\H{o}s and A.~C. Offord.
\newblock On the number of real roots of a random algebraic equation.
\newblock {\em Proc. London Math. Soc. (3)}, 6:139--160, 1956.

\bibitem{granville-wigman}
A.~Granville and I.~Wigman.
\newblock The distribution of the zeros of random trigonometric polynomials.
\newblock {\em Amer. J. Math.}, 133(2):295--357, 2011.

\bibitem{IZ}
I.~Ibragimov and D.~Zaporozhets.
\newblock On distribution of zeros of random polynomials in complex plane.
\newblock In {\em Prokhorov and contemporary probability theory}, volume~33 of
  {\em Springer Proc. Math. Stat.}, pages 303--323. Springer, Heidelberg, 2013.

\bibitem{ibragimov-maslova4}
I.~A. Ibragimov and N.~B. Maslova.
\newblock The average number of zeros of random polynomials.
\newblock {\em Vestnik Leningrad. Univ.}, 23(19):171--172, 1968.

\bibitem{ibragimov-maslova1}
I.~A. Ibragimov and N.~B. Maslova.
\newblock The average number of real roots of random polynomials.
\newblock {\em Dokl. Akad. Nauk SSSR}, 199:13--16, 1971.

\bibitem{ibragimov-maslova3}
I.~A. Ibragimov and N.~B. Maslova.
\newblock The mean number of real zeros of random polynomials. {I}.
  {C}oefficients with zero mean.
\newblock {\em Teor. Verojatnost. i Primenen.}, 16:229--248, 1971.

\bibitem{ibragimov-maslova2}
I.~A. Ibragimov and N.~B. Maslova.
\newblock The mean number of real zeros of random polynomials. {II}.
  {C}oefficients with a nonzero mean.
\newblock {\em Teor. Verojatnost. i Primenen.}, 16:495--503, 1971.

\bibitem{kallenberg}
O.~Kallenberg.
\newblock {\em Random measures}.
\newblock Akademie-Verlag, Berlin; Academic Press, Inc. [Harcourt Brace
  Jovanovich, Publishers], London, third edition, 1983.

\bibitem{konyagin-schlag}
S.~V. Konyagin and W.~Schlag.
\newblock Lower bounds for the absolute value of random polynomials on a
  neighborhood of the unit circle.
\newblock {\em Trans. Amer. Math. Soc.}, 351(12):4963--4980, 1999.

\bibitem{littlewood-offord1}
J.~E. Littlewood and A.~C. Offord.
\newblock On the {N}umber of {R}eal {R}oots of a {R}andom {A}lgebraic
  {E}quation.
\newblock {\em J. London Math. Soc.}, 13(4):288--295, 1938.

\bibitem{littlewood-offord2}
J.~E. Littlewood and A.~C. Offord.
\newblock On the number of real roots of a random algebraic equation. {III}.
\newblock {\em Rec. Math. [Mat. Sbornik] N.S.}, 12(54):277--286, 1943.

\bibitem{littlewood-offord4}
J.~E. Littlewood and A.~C. Offord.
\newblock On the distribution of the zeros and {$a$}-values of a random
  integral function. {I}.
\newblock {\em J. London Math. Soc.}, 20:130--136, 1945.

\bibitem{littlewood-offord3}
J.~E. Littlewood and A.~C. Offord.
\newblock On the distribution of zeros and {$a$}-values of a random integral
  function. {II}.
\newblock {\em Ann. of Math. (2)}, 49:885--952; errata 50, 990--991 (1949),
  1948.

\bibitem{MMrealroots}
M.~Michelen.
\newblock Real roots near the unit circle of random polynomials.
\newblock {\em In preparation}.

\bibitem{peres-virag}
Y.~Peres and B.~Vir\'{a}g.
\newblock Zeros of the i.i.d. {G}aussian power series: a conformally invariant
  determinantal process.
\newblock {\em Acta Math.}, 194(1):1--35, 2005.

\bibitem{rice1}
S.~O. Rice.
\newblock Mathematical analysis of random noise.
\newblock {\em Bell System Tech. J.}, 23:282--332, 1944.

\bibitem{rice2}
S.~O. Rice.
\newblock Mathematical analysis of random noise.
\newblock {\em Bell System Tech. J.}, 24:46--156, 1945.

\bibitem{shepp-vanderbei}
L.~A. Shepp and R.~J. Vanderbei.
\newblock The complex zeros of random polynomials.
\newblock {\em Trans. Amer. Math. Soc.}, 347(11):4365--4384, 1995.

\bibitem{tao-vu}
T.~Tao and V.~Vu.
\newblock Local universality of zeroes of random polynomials.
\newblock {\em International Mathematics Research Notices},
  2015(13):5053--5139, 2015.

\bibitem{sparoSur}
D.~I. \v{S}paro and M.~G. \v{S}ur.
\newblock On the distribution of roots of random polynomials.
\newblock {\em Vestnik Moskov. Univ. Ser. I Mat. Meh.}, 1962(3):40--43, 1962.

\end{thebibliography}
\end{document}